\documentclass[12pt]{amsart}
\usepackage{amssymb,amsmath,bm}
\usepackage{amsthm}
\usepackage{hyperref}
\usepackage{graphicx}        

\newcommand{\T}{\mathcal T}

\newtheorem{theorem}{Theorem}[section]

\newtheorem{proposition}[theorem]{Proposition}
\newtheorem{corollary}[theorem]{Corollary}
\newtheorem{conjecture}[theorem]{Conjecture}

\theoremstyle{definition}
\newtheorem{definition}[theorem]{Definition}

\theoremstyle{remark}
\newtheorem{remark}[theorem]{Remark}

\numberwithin{equation}{section}

\title{\bf Volume conjectures for the Reshetikhin-Turaev and the Turaev-Viro invariants}

\author{Qingtao Chen}
\address{Department of Mathematics, ETH Zurich, 
8092 Zurich, Switzerland}
\email{qingtao.chen@math.ethz.ch}

\author{Tian Yang}
\address{Department of Mathematics, Stanford University,
Stanford CA 94305, USA}
\email{yangtian@math.stanford.edu}

\thanks{The research of the first author was partially supported by the National Centre of Competence in Research SwissMAP of the Swiss National Science Foundation. The second author was supported by NSF grant DMS-1405066.}

\begin{document}

\maketitle

\begin{abstract}
We consider the asymptotics of the Turaev-Viro and the Reshetikhin-Turaev invariants of a  hyperbolic $3$-manifold, evaluated at the root of unity $\exp({2\pi\sqrt{-1}}/{r})$ instead of the standard $\exp({\pi\sqrt{-1}}/{r})$. We present evidence that, as $r$ tends to $\infty$, these invariants  grow exponentially with growth rates respectively given by the hyperbolic and the complex volume of the manifold. This reveals an asymptotic behavior that is different from that of Witten's Asymptotic Expansion Conjecture, which predicts polynomial growth of these invariants when evaluated at the standard root of unity. This new phenomenon suggests that the Reshetikhin-Turaev invariants may have a geometric interpretation other than the original one via $SU(2)$ Chern-Simons gauge theory. 
\end{abstract}


\section{Introduction}
In \cite{Wit89}, Witten provided a new interpretation of the Jones polynomial based on  Chern-Simons gauge theory, and expanded on this idea to construct a sequence of complex valued $3$-manifold invariants. This approach was formalized though the representation theory of quantum groups by Reshetikhin and Turaev \cite{RT90, RT91}, who generalized the Jones polynomial to a  sequence of polynomial invariants of a link, later called the colored Jones polynomials of that link. They also defined a sequence of $3$-manifold invariants  corresponding to Witten's invariants. The Reshetikhin-Turaev construction of 3-manifold invariants starts from a surgery description \cite{K78} of the manifold, and evaluates the colored Jones polynomials of the surgery data at certain roots of unity.

 A different approach was developed by Turaev and Viro \cite{TV92} who, from a triangulation of a closed 3-manifold,  constructed  real valued invariants of the manifold by using  quantum $6j$-symbols \cite{KR88}; these Turaev-Viro invariants turned out to be equal to the square of the norm of the Reshetikhin-Turaev invariants \cite{Ro, Tur10, W}. 

Using quantum dilogarithm functions, Kashaev \cite{Kas95, Kas97} used a different type of $6j$-symbols, involving the discrete quantum dilogarithm, to define for each integer $n$ complex valued link invariants. He observed in a few examples, and conjectured in the general case, that the absolute value of these invariants  grow exponentially with $n$, and that the growth rate is given by the hyperbolic volume of the complement of the link. In \cite{MM01}, Murakami and Murakami showed that Kashaev's invariants coincide with the values of the colored Jones polynomials at a certain root of unity, and reformulated Kashaev's conjecture as follows.

\theoremstyle{thm}
\newtheorem*{VolConj}{Volume Conjecture}

\begin{VolConj}
[\cite{Kas97,MM01}] \label{KMM}
For a hyperbolic link $L$ in $S^3$, let $J_n(L;q)$ be its $n$-th colored Jones polynomial. Then
$$\lim_{n\to+\infty}\frac{2\pi}{n}\log\big|J_n(L;e^{\frac{2\pi\sqrt{-1}}{n}})\big|=\mathrm{vol}(S^3\setminus L),$$
where $\mathrm{vol}(S^3\setminus L)$ is the hyperbolic volume of the complement of $L$.
\end{VolConj}

This conjecture has now been proved for a certain number of cases:  the figure-eight knot  \cite{MM01}; all hyperbolic knots with at most six crossings \cite{Oht1, Oht2}; the Borromean rings \cite{GL05}; the twisted Whitehead links \cite{Zhe07}; the Whitehead chains \cite{Vee08}. 
Various extensions of this conjecture have been proposed, and proved for certain cases in \cite{Guk05, MY07, Cos07',CGP, CGV, Mur13}.  


In the current paper we investigate the asymptotic behavior of the Reshetikhin-Turaev and the Turaev-Viro invariants, evaluated at the root of unity $q=e^{\frac{2\pi\sqrt{-1}}{r}}$. Supported by numerical  evidence, we propose the following conjecture.

\begin{conjecture}\label{VC} 
For a hyperbolic $3$-manifold $M,$ let $\mathrm{TV}_{r}(M; q)$ be its Turaev-Viro invariant and let $\mathrm{vol}(M)$ be its hyperbolic volume. Then for $r$ running over all odd integers and for $q=e^{\frac{2\pi\sqrt{-1}}{r}}$,$$\lim_{r\to+\infty}\frac{2\pi}{r}\log\big( \mathrm{TV}_{r}(M; q)\big)=\mathrm{vol}(M).$$
\end{conjecture}

 We here consider all types of hyperbolic $3$-manifolds: closed, cusped or those with totally geodesic boundary. The Turaev-Viro invariant  $\mathrm{TV}_{r}(M; q)$ is the original one defined in \cite{TV92} when the manifold $M$ is closed, and is its extension defined in \cite{BP} when $M$ has non-empty boundary. 
See \S \ref{definition} for  details. 

This conjecture should be contrasted with  Witten's Asymptotic Expansion Conjecture (see \cite{Oht01}) which predicts that, when evaluated at  $q=e^{\frac{\pi\sqrt{-1}}{r}}$,  the Witten invariants of a 3-manifold (and therefore its Reshetikhin-Turaev and  Turaev-Viro invariants) only grow  polynomially, with a growth rate related to  classical invariants of the manifold such as the Chern-Simons invariant  and the Reidemeister torsion.

Conjecture \ref{VC} is motivated by the beautiful work of Costantino \cite{Cos07} relating the asymptotics of quantum $6j$-symbols to the volumes of truncated hyperideal tetrahedra. See also \cite{CM, CGV}. 

We provide much supporting evidence for Conjecture \ref{VC}. In \S \ref{conjecture}, we numerically calculate $\mathrm{TV}_r(M)$ for various hyperbolic $3$-manifolds with cusps, including the figure-eight knot complement and its sister, the complements of the knots $K_{5_2}$ and $K_{6_1}$, and  the manifolds denoted by  $M_{3_6}$, $M_{3_8}$, $N_{1_1}$ and $N_{2_1}$ in the Callahan-Hildebrand-Weeks census \cite{CHW99}. We also numerically calculate $\mathrm{TV}_r(M)$  for the smallest hyperbolic $3$-manifolds with a totally geodesic boundary \cite{Fuj90, KM91}. 

Recently, Detcherry, Kalfagianni and the second author \cite{DKY} provided a rigorous proof of Conjecture \ref{VC} for the figure-eight knot complement.

\medskip

The Reshetikhin-Turaev invariants $\mathrm{RT}_r(M;q)$ are complex valued invariants of a closed oriented  $3$-manifold $M$, defined for all integers $r\geqslant 3$ and  all primitive $2r$-th roots of unity $q$.  For $q=e^{\frac{\pi\sqrt{-1}}{r}}$,  these invariants provide a mathematical realization of Witten's invariants \cite{Wit89}. Following a skein theory approach pioneered by Lickorish \cite{Li91, Li92},  Blanchet-Habegger-Masbaum-Vogel \cite{BHMV} (see also Lickorish \cite{Li}) extended Reshetikhin-Turaev invariants to  primitive $r$-th roots of unity $q$ with $r$ odd. In particular, $\mathrm{RT}_r(M;q)$ is defined at $q=e^{\frac{2\pi\sqrt{-1}}{r}}$ when $r$ is odd. In \S \ref{closed}, we numerically compute  Reshetikhin-Turaev invariants for various  closed  hyperbolic $3$-manifolds obtained by  integral Dehn surgery along the knots $K_{4_1}$ and $K_{5_2}$. These calculations suggest the following conjecture. 

\begin{conjecture}\label{RTVC} Let $M$ be a closed oriented hyperbolic $3$-manifold and let $\mathrm{RT}_r(M;q)$ be its Reshetikhin-Turaev invariants. Then for $q=e^{\frac{2\pi\sqrt{-1}}{r}}$ with $r$ odd and for a suitable choice of the arguments, $$\lim_{r\to+\infty}\frac{4\pi \sqrt{-1}}{r}\log\big( \mathrm{RT}_{r}(M; q)\big)= \mathrm{CS}(M)+\mathrm{vol}(M)\sqrt{-1} \mod \pi^2\mathbb Z,$$
where $\mathrm{CS}(M)$ denotes the  Chern-Simons invariant of the hyperbolic metric of $M$ multiplied by~$2\pi^2$.
\end{conjecture}

Ohtsuki  \cite{Oht3} recently announced a proof of Conjecture \ref{RTVC} for the manifolds obtained by Dehn surgery along the knot $K_{4_1}$. By \cite{Ro, Tur10, W},  Conjecture \ref{RTVC} implies Conjecture \ref{VC} for  closed $3$-manifolds.
 
 Comparing Conjecture \ref{RTVC} with Witten's Asymptotic Expansion Conjecture, one sees a very different asymptotic behavior for the Reshetikhin-Turaev invariants evaluated at $q=e^{\frac{2\pi\sqrt{-1}}{r}}$ and $q=e^{\frac{\pi\sqrt{-1}}{2r}}$. Our numerical calculations  also suggest exponential growth at other roots of unity such as $q=e^{\frac{3\pi\sqrt{-1}}{r}}$. For these roots of unity, we expect a geometric interpretation of Reshetikhin-Turaev invariants that is different from  the $SU(2)$ Chern-Simons gauge theory. 

\medskip

In \S \ref{unknot}, we calculate $\mathrm{TV}_r(M)$ for the complements of the unknot,  the Hopf link, the trefoil knot and the torus links $T_{(2,4)}$ and $T_{(2,6)}$. 
We also numerically calculate $\mathrm{TV}_r(M)$ for the complement of the torus knots $T_{(2,5)}$, $T_{(2,7)}$, $T_{(2,9)}$, $T_{(2,11)}$, $T_{(3,5)}$ and $T_{(3,7)}$. These computations suggest an Integrality Conjecture (Conjecture \ref{int}) which states that the Turaev-Viro invariants of  torus link complement are integers independent of the roots of unity at which they are evaluated. 
\\

\noindent\textbf{Acknowledgments:} Part of this work was done during the Tenth East Asian School of Knots and Related Topics at the East China Normal University in January 2015. We would like to thank the organizers for their support and hospitality.

The authors are deeply grateful to Francis Bonahon for discussions, suggestions and improving the writing of the paper, and are also grateful to Riccardo Benedetti, Francesco Costantino, Charles Frohman, Stavros Garoufalidis,  Rinat Kashaev,  Liang Kong, Thang L\^e,  Julien March\'e, Gregor Masbaum, Nicolai Reshetikhin, Dylan Thurston, Roland van der Veen, Zhenghan Wang and Hao Zheng for discussions and suggestions, and  Xiaogang Wen and Shing-Tung Yau for showing interest in this work. The first author thanks Kefeng Liu and Weiping Zhang for many useful discussions during the past few years, and  Nicolai Reshetikhin for guiding him to the area of quantum invariants and sharing many experiences and ideas since 2005. The second author thanks Henry Segerman and Hongbin Sun for very helpful discussions, and Steven Kerckhoff, Feng Luo and Maryam Mirzakhani  for  many suggestions.


\section{Preliminaries}

We recall the construction of Turaev-Viro invariants of 3-dimensional manifolds. In order to follow a uniform treatment for all cases, we extend their definition to pseudo $3$-manifolds. 

\subsection{Pseudo $3$-manifolds and triangulations}\label{ideal}

A \emph{pseudo $3$-manifold} is a topological space $M$ such that each point $p$ of $M$ has a neighborhood $U_p$ that is homeomorphic to a cone over a surface $\Sigma_p$. We call $p$ a \emph{singular point} and $U_p$ a \emph{singular neighborhood} if $\Sigma_p$ is not a $2$-sphere. In particular, a closed $3$-manifold is a pseudo $3$-manifold with no singular point, and every $3$-manifold with boundary is homeomorpic to a pseudo $3$-manifold with suitable singular neighborhoods of all singular points removed.

A \emph{triangulation} $\mathcal{T}$ of a pseudo manifold $M$ consists of a disjoint union $X=\bigsqcup \Delta_i$ of finitely many  Euclidean tetrahedra $\Delta_i$ and of a collection of  homeomorphisms $\Phi$ between pairs of  faces in $X$ such that the quotient space $X/\Phi$ is homeomorphic to $M$. The \emph{vertices}, \emph{edges}, \emph{faces} and \emph{tetrahedra} in $\mathcal{T}$ are respectively the quotients of the vertices, edges, faces and tetrahedra in $X$.  From the definition, we see that a singular point of $M$ must be a vertex of $\T$. We call the non-singular vertices of $\T$ the \emph{inner vertices}. If $M$ is a closed $3$-manifold, then a triangulation of $M$ is a triangulation of manifold in the usual sense; and if $N$ is a $3$-manifold with boundary obtained by removing all singular neighborhoods of a pseudo $3$-manifold $M$, then a triangulation of $M$ without inner vertices determines an ideal triangulation of $N$.

In \cite{Mat87, Mat07, Pie88}, it is proved that any two triangulations of a pseudo $3$-manifold are related by a sequence of 0--2 and 2--3 Pachner moves. See the figure below, where in the $0-2$ move a new inner vertex is introduced.

\centerline{\includegraphics[width=14cm]{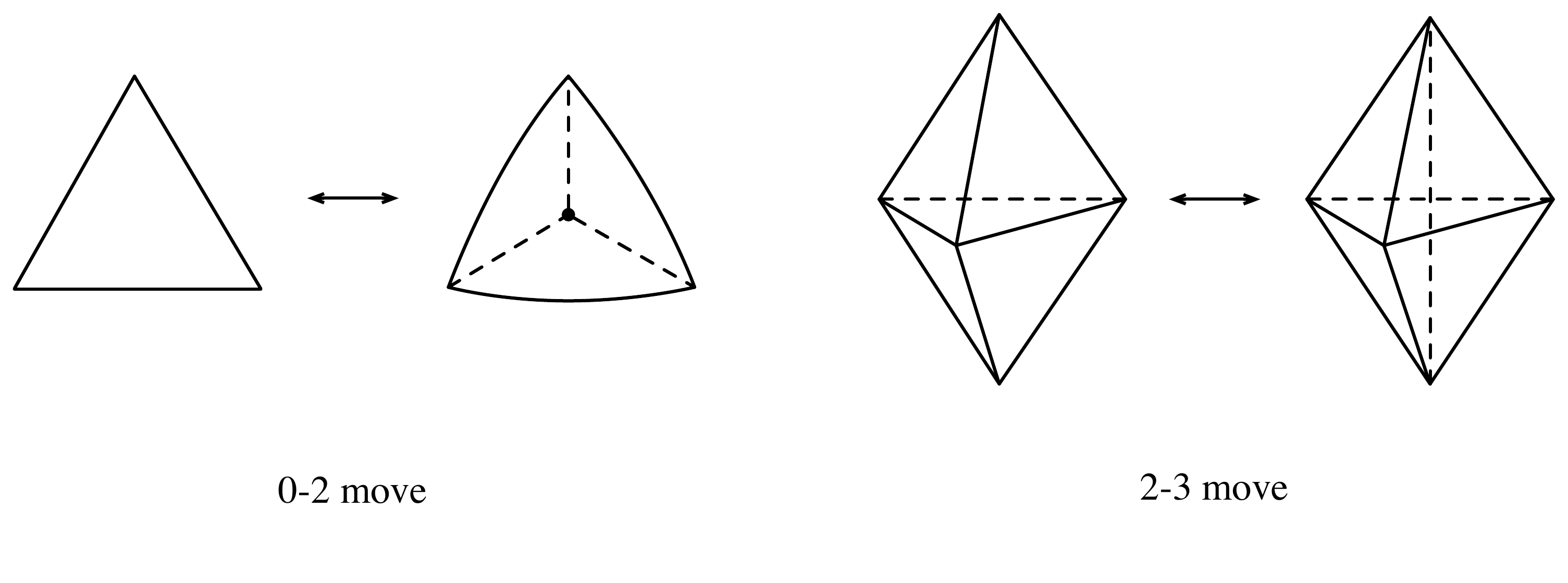}}


\subsection{Quantum $6j$-symbols}
\label{q6j}

We now recall the definition and basic properties of the quantum $6j$-symbols. See \cite{KR88, KL94} for more details.

Throughout this subsection, we will fix an integer $r\geqslant 3$, and we let $I_r=\{0,1/2,\dots, (r-2)/2\}$ be the set of non-negative half-integers less than or equal to $(r-2)/2$.  The elements of $I_r$ are traditionally called \emph{colors}. 

Let $q\in \mathbb C$ be a root of unity such that $q^2$ is a primitive root of unity of order $r$. For an integer $n$, the \emph{quantum integer} $[n]$ is the real number defined by
$$[n]=\frac{q^n-q^{-n}}{q-q^{-1}},$$ 
and the associated \emph{quantum factorial} is
$[n]!=[n][n-1]\dots[1]$. By convention, $[0]!=1$.

A triple $(i,j,k)$ of elements of $I_r$ is called \emph{admissible} if 
\begin{enumerate}
\item  $i+j\geqslant k$, $j+k\geqslant i$ and $k+i\geqslant j$, \item  $i+j+k\in \mathbb Z$, 
\item $i+j+k\leqslant r-2$. 
\end{enumerate}
A  6-tuple $(i,j,k,l,m,n)$ of elements of $I_r$ is \emph{admissible}  if the triples $(i,j,k)$, $(j,l,n)$, $(i,m,n)$ and $(k,l,m)$ are admissible

For an admissible triple $(i,j,k)$, define 
$$\Delta(i,j,k)=\sqrt{\frac{[i+j-k]![j+k-i]![k+i-j]!}{[i+j+k+1]!}}$$
with the convention that $\sqrt{x}=\sqrt{|x|}\sqrt{-1}$ when the real number $x$ is negative.

\begin{definition}
The \emph{quantum $6j$-symbol} of an admissible 6-tuple $(i,j,k,l,m,n)$ is the number
\begin{multline*}
\bigg|\begin{matrix} i & j & k \\l & m & n \end{matrix} \bigg|
= \sqrt{-1}^{-2(i+j+k+l+m+n)}\Delta(i,j,k)\Delta(j,l,n)\Delta(i,m,n)\Delta(k,l,m)\\
\sum_{z=\max \{T_1, T_2, T_3, T_4\}}^{\min\{ Q_1,Q_2,Q_3\}}\frac{(-1)^z[z+1]!}{[z-T_1]![z-T_2]![z-T_3]![z-T_4]![Q_1-z]![Q_2-z]![Q_3-z]!}
\end{multline*}
where $T_1=i+j+k$, $T_2=j+l+n$, $T_3=i+m+n$ and $T_4=k+l+m$, $Q_1=i+j+l+m$, $Q_2=i+k+l+n$ and $Q_3=j+k+m+n$.
\end{definition}

A good way to memorize the definitions is to consider a tetrahedron as in  Figure~\ref{fig:tetra}, and to attach the weights $i$, $j$, $k$, $l$, $m$, $n$ to its edges as indicated in the figure. Then each of $T_1$, $T_2$, $T_3$, $T_4$ corresponds to a face of the tetrahedron, and each of $Q_1$, $Q_2$, $Q_3$ corresponds to a quadrilateral separating two pairs of the vertices. 

\begin{figure}[htbp]

\includegraphics[width=4cm]{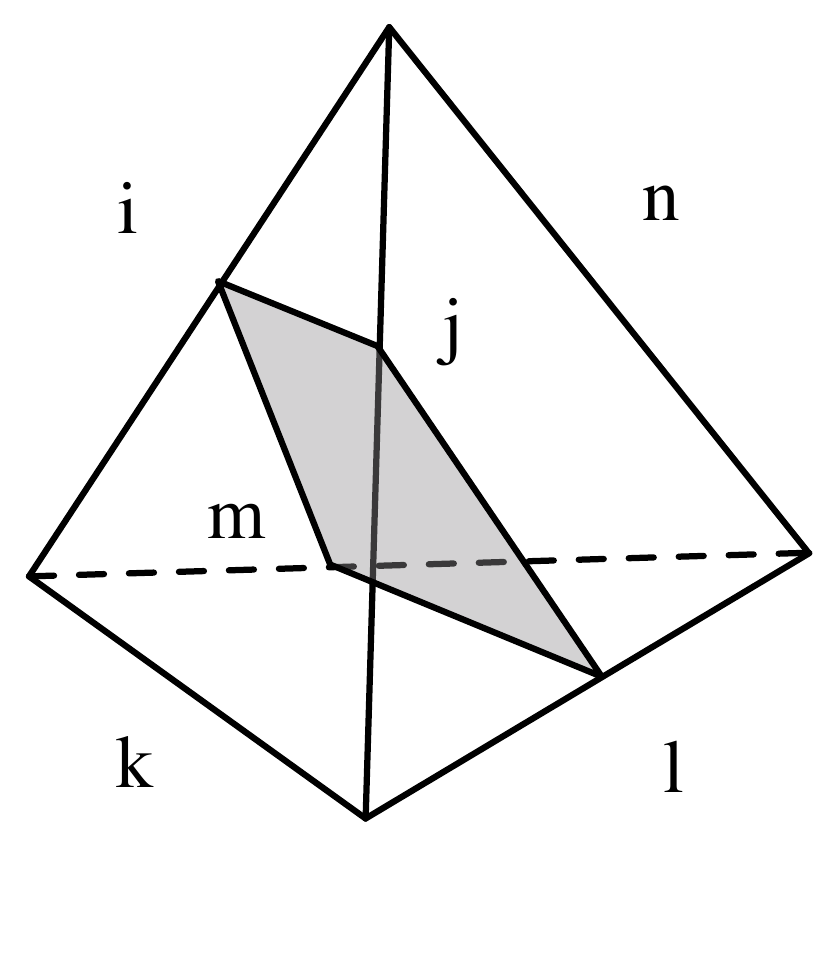}
\caption{$6j$-symbols and the tetrahedron}
\label{fig:tetra}
\end{figure}

The following symmetries
\begin{equation*}
\bigg|
\begin{matrix}
        i & j & k \\
        l & m & n 
      \end{matrix} \bigg|
      =\bigg|
      \begin{matrix}
        j & i & k \\
        m & l & n 
      \end{matrix} \bigg|=\bigg|
      \begin{matrix}
        i & k & j \\
        l & n & m 
      \end{matrix} \bigg|
      =\bigg|
      \begin{matrix}
        i & m & n \\
        l & j & k 
      \end{matrix} \bigg|=\bigg|
      \begin{matrix}
        l & m & k \\
        i & j & n
      \end{matrix} \bigg|=\bigg|
      \begin{matrix}
        l & j & n \\
       i & m & k 
      \end{matrix} \bigg|
\end{equation*}  
immediately follow from the definitions.

The quantum $6j$-symbols satisfy the following two important identities, which are crucial in the construction of the Turaev-Viro invariants. For $i\in I_r$, set
$$w_i=(-1)^{2i}[2i+1]\qquad
\text{ and }\qquad
\eta=\sum_{i\in I_r}w_i^2.$$

\begin{proposition} [Orthogonality Property] For any admissible $6$-tuple $(i,j,k,l,m,n)$, 
\begin{equation}\label{O}
\sum_sw_sw_m\bigg|
      \begin{matrix}
        i & j & m\\
        k & l & s 
      \end{matrix} \bigg|\bigg|
      \begin{matrix}
        i & j & n \\
        k & l & s 
      \end{matrix} \bigg|=\delta_{mn},
      \end{equation}
where $\delta$ is the Kronecker symbol, and where the sum is over all $s\in I_r$ such that the two $6$-tuples  in the sum are admissible. \qed
\end{proposition} 

\begin{corollary}\label{C}
For any admissible triple $(i,j,k)$,
\begin{equation}
\eta^{-1}\sum_{l,m,n}w_lw_mw_n\bigg|
      \begin{matrix}
        i & j & k\\
        l & m & n 
      \end{matrix} \bigg|\bigg|
      \begin{matrix}
        i & j & k \\
        l & m & n 
      \end{matrix} \bigg|=1,
      \end{equation}
where the sum is over $l,m,n\in I_r$ such that the $6$-tuples $(i,j,k,l,m,n)$ is admissible. \qed

\end{corollary}

\begin{proposition} [Biedenharn-Elliot identity] For any $i$, $j$, $k$, $l$, $m$, $n$, $o$, $p$, $q\in I_r$ such that $(o,p,q,i,j,k)$ and $(o,p,q,l,m,n)$ are admissible, 

\begin{equation}\label{BE}
\sum_sw_s\bigg|
      \begin{matrix}
        i & j & q \\
        m & l & s 
      \end{matrix} \bigg|\bigg|
      \begin{matrix}
        j & k & o \\
        n & m & s 
      \end{matrix} \bigg|\bigg|
      \begin{matrix}
        k & i & p \\
        l & n & s 
      \end{matrix} \bigg|=\bigg|
      \begin{matrix}
        o & p & q \\
        i & j & k 
      \end{matrix} \bigg|\bigg|
      \begin{matrix}
        o & p & q \\
        l & m & n 
      \end{matrix} \bigg|,
      \end{equation}
where the sum is over $s\in I_r$ such that the three $6$-tuples in the sum are admissible. \qed

\end{proposition}


\subsection{Turaev-Viro invariants of pseudo $3$-manifolds}\label{definition}
Let $q$ be a root of unity, and let $r$ be such that $q^2$ is a primitive root of unity of order $r$. As in \S \ref{q6j}, we consider the set $I_r=\{0,1/2,1,\dots,(r-2)/2\}$ of colors, and the notation
$$
[n] = \frac {q^n - q^{-n}}{q-q^{-1}},
\qquad
w_i=(-1)^{2i}[2i+1],
\qquad
\eta=\sum_{i\in I_r}w_i^2.
$$
for every integer $n$ and color $i \in I_r$.

For a triangulation $\mathcal T$ of a  pseudo $3$-manifold $M$, an \emph{$r$-admissible coloring} of $(M,\mathcal T)$ is a map
$$
c \colon \{\text{edges of } \mathcal T\} \to I_r
$$
such that, for every 2-dimensional face $F$ of $\mathcal T$, the colors $c(e_1)$, $c(e_2)$, $c(e_3) \in \mathcal T$ associated to the  edges of $F$ form an admissible triple. Such a coloring $c$ associates to each edge $e$ of $\mathcal T$ the number 
$$|e|_c= w_{c(e)},$$
and to each tetrahedron $\Delta$ of $\mathcal T$ the $6j$-symbol 
\begin{equation*}
|\Delta|_c=\bigg|
      \begin{matrix}
        c(e_{12}) & c(e_{13}) & c(e_{23}) \\
       c(e_{34}) & c(e_{24}) & c(e_{14}) \\
      \end{matrix} \bigg|,
\end{equation*}
where the edges of $\Delta$ are indexed in such a way that, if $v_1$, $v_2$, $v_3$, $v_4$ denote the vertices of $\Delta$, the edge $e_{ij}$ connects $v_i$ to $v_j$.

\begin{definition}\label{TV} With the above definitions,  the \emph{Turaev-Viro invariant} of $M$ associated to the root of unity $q$ is defined as the sum
 $$\mathrm{TV}_q(M,\T)=\eta^{-|V|} \sum_{c\in A_r}\prod_{e\in E}|e|_c\prod_{\Delta\in T}|\Delta|_c$$
 where $V$, $E$, $T$, $A_r$ respectively denote the sets of inner vertices, edges, tetrahedra and $r$-admissible colorings of the triangulation $\mathcal T$.
\end{definition}

\begin{theorem}\label{TVinvariance}
The above invariant $\mathrm{TV}_q(M,\T)$ depends only on the pseudo-manifold $M$ and on the root of unity $q$, not on the triangulation $\mathcal T$. 
\end{theorem}

\begin{proof} Theorem~\ref{TVinvariance} is proved by a straightforward extension to pseudo 3-manifolds of the original argument of Turaev and Viro in  \cite{TV92} for 3-manifolds. 

The first ingredient is a purely topological statement, proved in  \cite{Mat87, Mat07, Pie88}, which says  that any two triangulations of a pseudo $3$-manifold are related by a sequence of the Pachner Moves 0--2 and 2--3 represented in Figures~\ref{fig:02} and \ref{fig:23}. The Pachner Move 0--2 replaces a 2-dimensional face of the triangulation by two tetrahedra meeting along 3 faces, and adds one vertex to the triangulation. The 2--3 Move replaces two tetrahedra meeting  along one face by three tetrahedra sharing one edge. 

\begin{figure}[htbp]

\includegraphics[width=8.5cm]{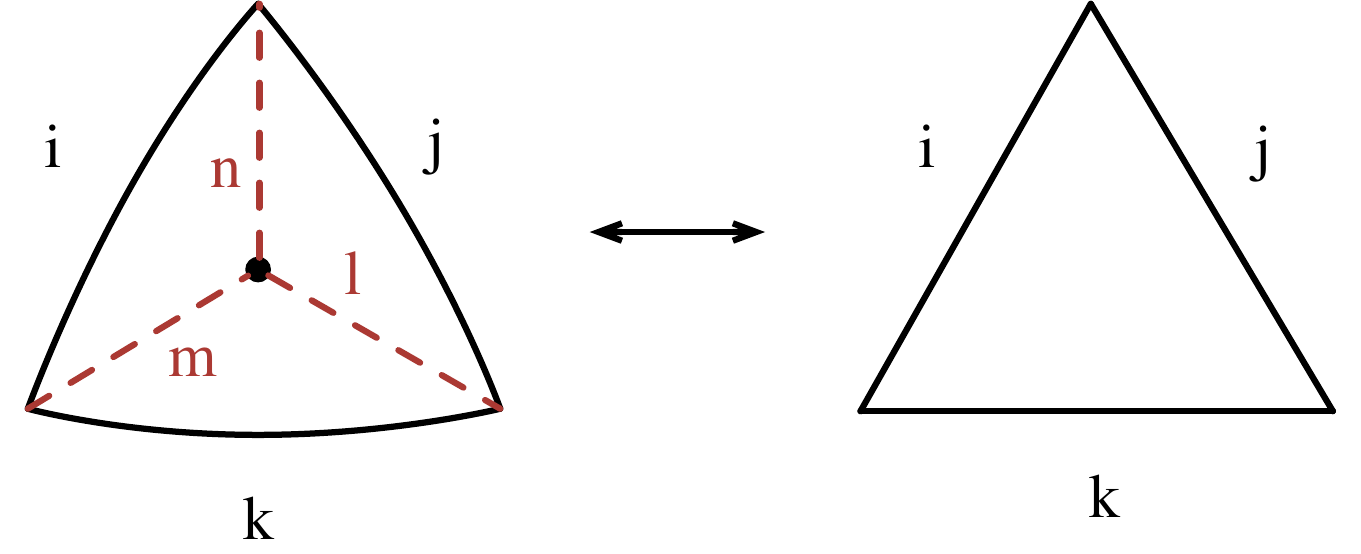}

\caption{The Pachner Move 0--2}
\label{fig:02}
\end{figure}
\begin{figure}[htbp]
\includegraphics[width=8.5cm]{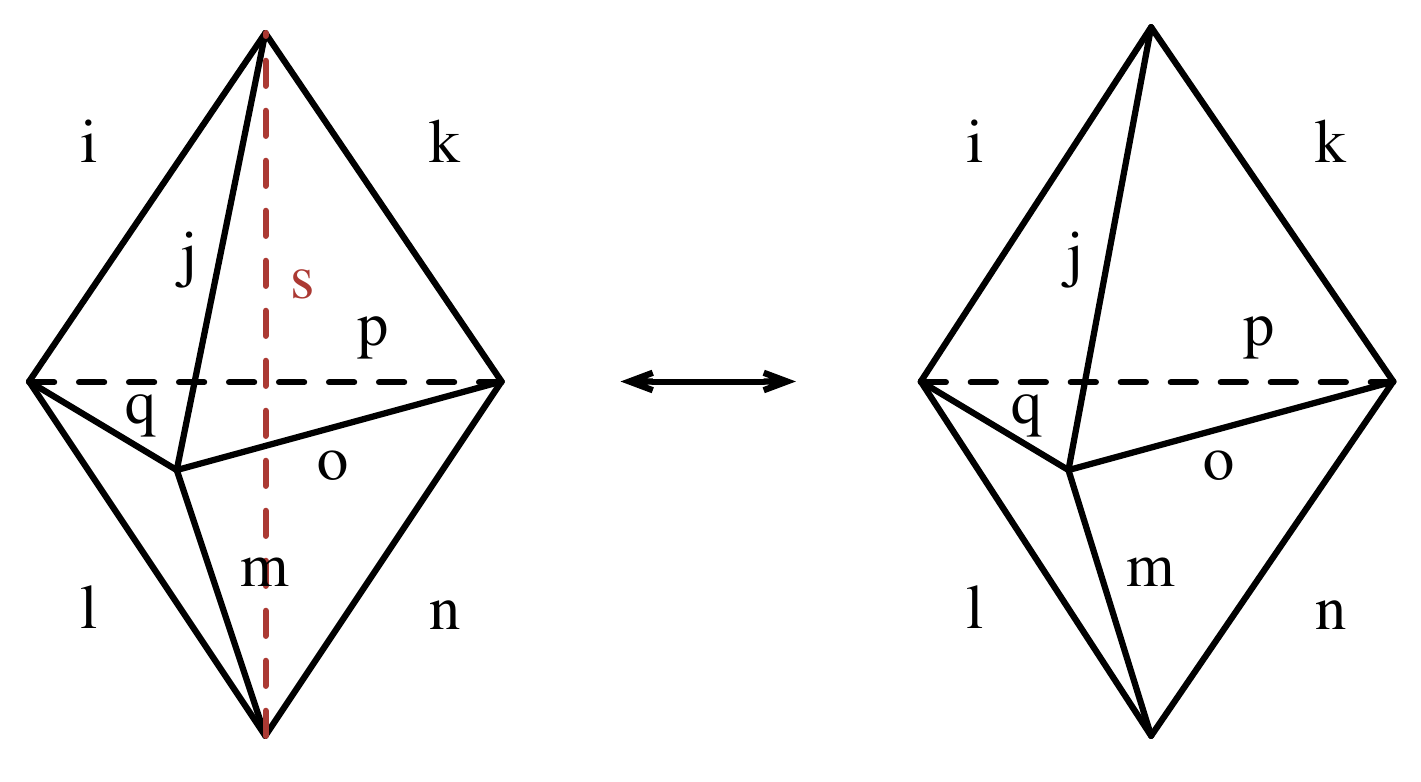}
\caption{The Pachner Mover 2--3}
\label{fig:23}
\end{figure}

The second ingredient is algebraic, and is provided by the properties of $6j$-symbols given in \S \ref{q6j}. Indeed, exactly as in \cite{TV92}, Corollary \ref{C} of the Orthogonality Property of Proposition~\ref{O} shows that $\mathrm{TV}_r(M,\T)$ is unchanged as we modify the triangulation $\mathcal T$ by a  0--2 move, and the Biedenharn-Elliot identity (\ref{BE}) guarantees the invariance under the 2--3 move. 
\end{proof}

As mentioned in \S \ref{ideal}, a triangulation of a pseudo $3$-manifold without inner vertices determines an ideal triangulation of the $3$-manifold with boundary obtained by removing all the singular neighborhoods. Hence for a $3$-manifold $M$ with boundary, one can define $\mathrm{TV}_r(M)$ using an ideal triangulation of $M$. Our invariant (and its construction) then coincides with the one defined in \cite{BP} using o-graphs.

Theorem \ref{TVinvariance} shows that, for any $r$ and $q$ as above,  $\mathrm{TV}_q(M,\T)$ is  independent of the choice of of the triangulation $\mathcal T$. We will consequently   omit the triangulation $\T$ and denote the invariant by $\mathrm{TV}_r(M)$ if $q=e^{\frac{2\pi\sqrt{-1}}{r}}$, or by $\mathrm{TV}_r(M; q)$ if we want to emphasize which root of unity $q$ is being used.


\section{Evidence for $3$-manifolds with boundary}\label{conjecture}

We now provide numerical evidence for Conjecture \ref{VC} for a few hyperbolic $3$-manifolds with boundary. The closed manifold case will be considered in the next section. The reason for considering the two cases separately is that a hyperbolic $3$-manifold with boundary often admits an  ideal triangulation by a small number of tetrahedra, whereas a triangulation of a closed hyperbolic $3$-manifold usually requires more tetrahedra. For example, it takes at least nine tetrahedra to triangulate the Weeks manifold, which is the smallest closed hyperbolic $3$-manifold.

To simplify the notation, set 
$$QV_r(M)=\frac{2\pi}{r-2}\log\Big(\mathrm{TV}_r(M;e^{\frac{2\pi\sqrt{-1}}{r}})\Big)$$ 
for each odd integer $r\geqslant 3$. Similarly, write
$$\mathrm{TV}_r(L)=\mathrm{TV}_r(S^3\setminus L)\quad\text{and}\quad QV_r(L)=QV_r(S^3\setminus L)$$
when $M=S^3\setminus L$ is a link complement. 

\subsection{The figure-eight knot complement and its sister}\label{41}

\bigskip

By Thurston's famous construction \cite{Thu82}, the figure-eight knot complement $S^3\setminus K_{4_1}$ has volume 
$$\mathrm{vol}(S^3\setminus K_{4_1})\approx2.02988,$$
and has the  ideal triangulation represented  in Figure~\ref{fig:FigureEight}. In that figure, edges with the same labels ($a$ or $b$) are glued together following the indicated orientations  to form an edge of the ideal triangulation.

\begin{figure}[htbp]

\includegraphics[width=8cm]{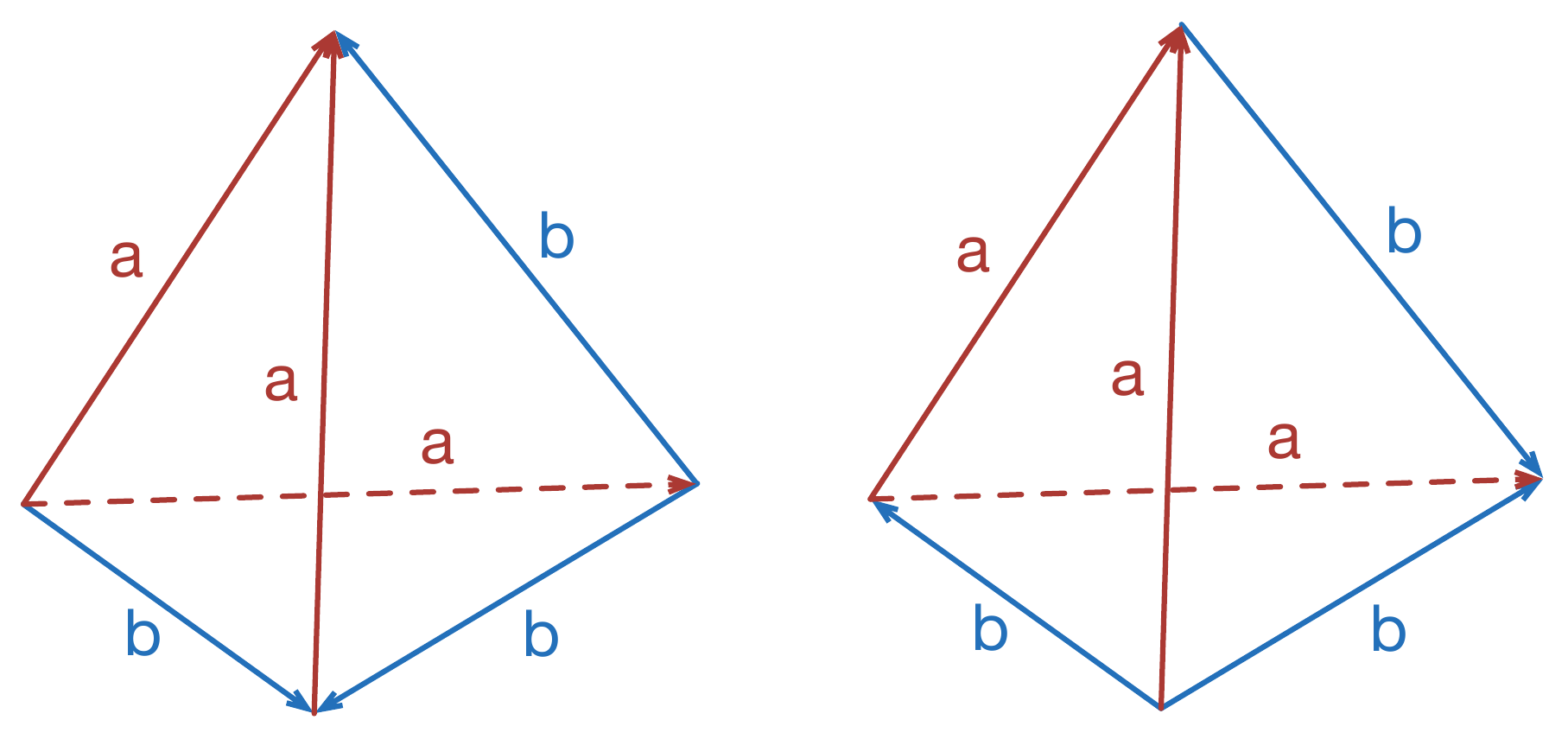}

\caption{}
\label{fig:FigureEight}
\end{figure}

 By Definition \ref{TV}, we have 
\begin{equation*}
\mathrm{TV}_r(K_{4_1})=\sum_{(a,b)\in A_r} w_aw_b
\bigg|
      \begin{matrix}
        a & a & b \\
        a & b & b \\
      \end{matrix} \bigg|
      \bigg|
      \begin{matrix}
        a & a & b \\
        a & b & b \\
      \end{matrix} \bigg|,
      \end{equation*}
where $A_r$ consists of the pairs $(a,b)$ of elements of $I_r$ such that $(a,a,b)$ and $(b,b,a)$ are admissible, i.e., $2a-b\geqslant 0$, $2b-a\geqslant 0$, $2a+b\leqslant r-2$, $2b+a\leqslant r-2$ and $2a+b$ and $2b+a$ are integers. From this formula, we have the following table of values of $QV_r(K_{4_1})$. \bigskip

\centerline{\tiny\begin{tabular}{|c||c|c|c|c|c|c|c|c|c|} 
\hline
 r &   $11$ & $13$ & $15$ & $17$ & $19$ & $21$ & $23$ & $25$ & $31$  \\ 
\hline
&&&&&&&&&\\
\ \ $QV_r(K_{4_1})$ \ \ & \ \  $2.40661$  \ \   & \ \ $2.37755$ \ \ &\ \ $2.34826$ \ \ &\ \ $2.31907$ \ \ &\ \  $2.29953$ \ \ & \ \ $2.28227$ \ \ &\ \  $2.26834$ \ \ &\ \  $2.25634$ \ \   & \ \ $2.22824$ \ \  \\
\hline
\end{tabular}}
\smallskip

\centerline{\tiny\begin{tabular}{|c||c|c|c|c|c|c|c|c|c|} 
\hline
$ r$  &$41$ & $51$ & $61$ & $71$ & $81$ & $91$ & $101$  &  $111$ & $121$     \\ 
\hline
&&&&&&&&&\\
\ \ $QV_r(K_{4_1})$ \ \ & \ \  $2.19685$ \ \ & \ \ $2.17540$ \ \   &   \ \ $2.15953$ \ \ & \ \ $2.14721$  \ \  & \ \ $2.13731$ \ \ & \ \ $2.12915$ \ \ & \ \ $2.12230$  \ \ & \ \ $2.11643$ \ \ & \ \ $2.11136$  \ \ \\
\hline
\end{tabular}}
\smallskip

\centerline{\tiny\begin{tabular}{|c||c|c|c|c|c|c|c|c|c|} 
\hline
$ r$   & $131$ & $141$ & $151$ & $201$ & $301$ & $401$ & $501$  &  $701$ & $1001$    \\ 
\hline
&&&&&&&&&\\
\ \ $QV_r(K_{4_1})$ \ \ & \ \ $2.10692$ \ \ & \ \ $2.10299$ \ \   &   \ \ $2.09949$ \ \ & \ \ $2.08641$  \ \  & \ \ $2.07168$ \ \ & \ \ $2.06344$ \ \ & \ \ $2.05810$  \ \ & \ \ $2.05153$ \ \ & \ \ $2.04614$ \ \ \\
\hline
\end{tabular}}
\bigskip

Figure \ref{Fig1} below  compares the values of the Turaev-Viro invariants  $QV_r(K_{4_1})$ and the Kashaev invariants $\langle K_{4_1}\rangle_r$ for various values of $r$. The dots represent the points $(r,QV_r(K_{4_1}))$, the diamonds represent the points $(r, \frac{2\pi}{r}\log|\langle K_{4_1}\rangle_r|)$, and the squares represent the points $(r, \mathrm{vol}(S^3\setminus K_{4_1}))$. Note that the values of $QV_r(K_{4_1})$ appear to converge to $\mathrm{vol}(S^3\setminus K_{4_1})$ much faster than $\langle K_{4_1}\rangle_r$ as $r$ becomes large.

\begin{figure}[htbp]
\centering
\includegraphics[scale=0.6]{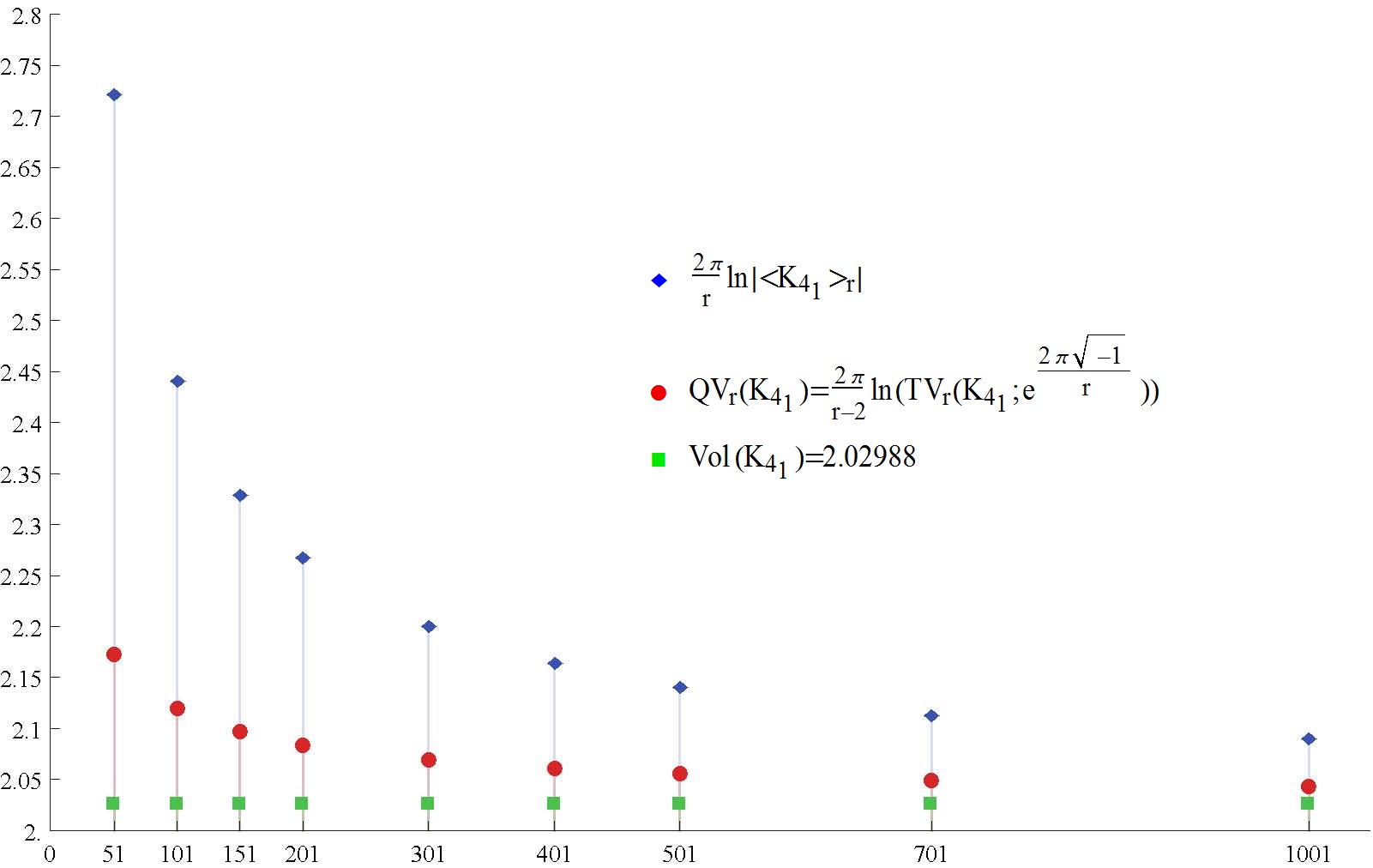}
\caption{Comparison of different invariants for $K_{4_1}$}
\label{Fig1}
\end{figure}

The manifold $M_{2_2}$ in the Callahan-Hildebrand-Weeks census \cite{CHW99}, also known as the figure-eight sister, shares the same volume with the figure-eight knot complement, i.e.,
$$\mathrm{vol}(M_{2_2})=\mathrm{vol}(S^3\setminus K_{4_1})\approx2.02988.$$
It is also known that $M_{2_2}$ is not the complement of any knot in $S^3$. According to Regina \cite{Reg}, $M_{2_2}$ has the  ideal triangulation represented in Figure~\ref{fig:FigEightSister}.

\begin{figure}[htbp]

\includegraphics[width=8cm]{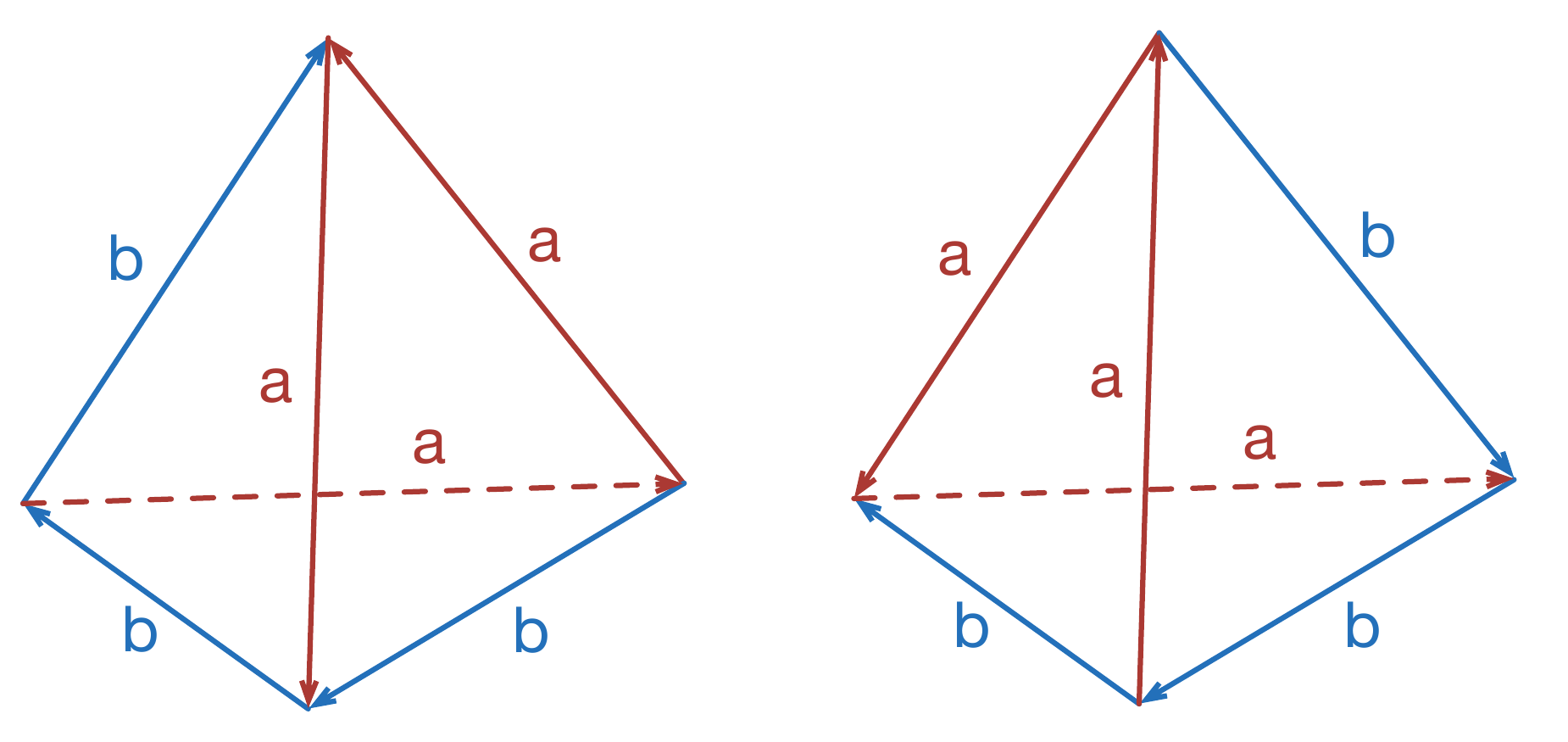}

\caption{}
\label{fig:FigEightSister}
\end{figure}

Since for each tetrahedron in this triangulation, the coloring is the same as that of $S^3\setminus K_{4_1}$, the invariant $\mathrm{TV}_r(M_{2_2})$ has exactly the same formula as $\mathrm{TV}_r(K_{4_1})$. As a consequence, the Turaev-Viro invariants of these manifolds take the same values.


\subsection{The $K_{5_2}$ knot complement and its sisters}\label{52}

According to SnapPy \cite{Sna} and Regina \cite{Reg}, the complement of the knot $K_{5_2}$  has volume $$\mathrm{vol}(S^3\setminus K_{5_2})\approx2.82812,$$ and admits the  ideal triangulation represented in Figure~\ref{fig:K52}. Since only the colors of the edges (according to which the edges are identified to form an edge of the triangulation) matters in the calculation of $\mathrm{TV}_r(M)$, we omit the arrows on the edges.
\begin{figure}[htbp]
\includegraphics[width=12cm]{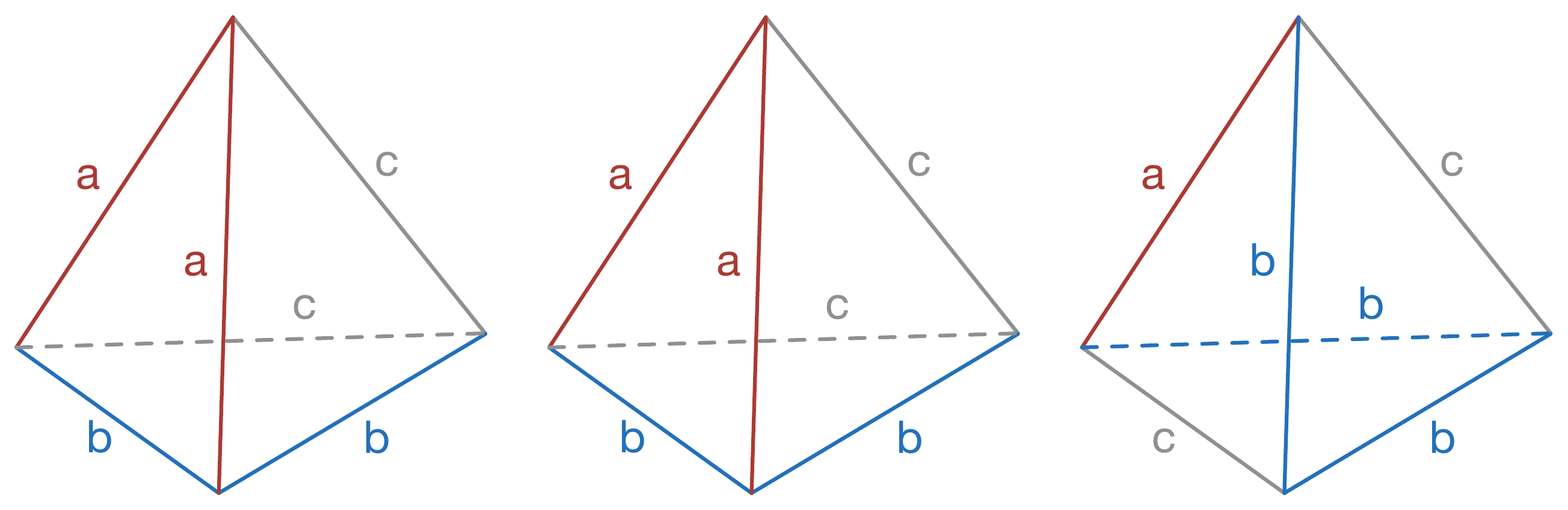}

\caption{}
\label{fig:K52}
\end{figure}

By Definition \ref{TV}, we have
\begin{equation*}
\mathrm{TV}_r(K_{5_2})=\sum_{(a,b,c)\in A_r} w_aw_bw_c
\bigg|
      \begin{matrix}
        a & a & b \\
        b & c & c \\
      \end{matrix} \bigg|\bigg|
      \begin{matrix}
        a & a & b \\
        b & c & c \\
      \end{matrix} \bigg|
      \bigg|
      \begin{matrix}
        a & b & c \\
        b & b & c \\
      \end{matrix} \bigg|,
      \end{equation*}
where $A_r$ consists of triples $(a,b,c)$ of elements of $I_r$ such that $(a,a,b)$, $(b,b,c)$, $(c,c,a)$ and $(a,b,c)$ are admissible. From this, we have the following table of values of $QV_r(K_{5_2})$. 
\bigskip

\centerline{\tiny\begin{tabular}{|c||c|c|c|c|c|c|c|c|} 
\hline
 r &$7$ & $9$ & $11$ & $21$ & $31$ & $41$   & $51$  & $61$ \\ 
\hline
&&&&&&&&\\
\ \ $QV_r(K_{5_2})$ \ \ & \ \ $3.38531$ \ \ &\ \ $3.32394$ \ \ &\ \ $3.25282$ \ \ &\ \  $3.09588$ \ \ & \ \ $3.03657$ \ \ &\ \  $3.00236$ \ \  &\ \  $2.97925$ \ \  &  \ \ $2.96232$ \ \  \\
\hline
\end{tabular}}
\smallskip

\centerline{\tiny\begin{tabular}{|c||c|c|c|c|c|c|c|c|} 
\hline
$ r$ & $71$ & $81$ & $91$ & $101$ & $121$ & $151$ & $201$  & $301$ \\ 
\hline
&&&&&&&&\\
\ \ $QV_r(K_{5_2})$ \ \ & \ \ $2.94927$ \ \ & \ \ $2.93883$ \ \ & \ \ $2.93027$ \ \ & \ \ $2.92309$ \ \  & \ \ $2.91169$ \ \ & \ \ $2.89937$ \ \ & \ \ $2.88586$  \ \ & \ \ $2.87071$  \ \   \\
\hline
\end{tabular}}
\bigskip

Figure \ref{Fig2}   compares the values of $QV_r(K_{5_2})$ with those of the Kashaev invariants $\langle K_{5_2}\rangle_r$. Again,  the dots represent the points $(r,QV_r(K_{5_2}))$,  the diamonds represent the points $(r, \frac{2\pi}{r}\log|\langle K_{5_2}\rangle_r|)$, and the squares represent the points $(r, \mathrm{vol}(S^3\setminus K_{5_2}))$.

\begin{figure}[htbp]
\centering
\includegraphics[scale=0.6]{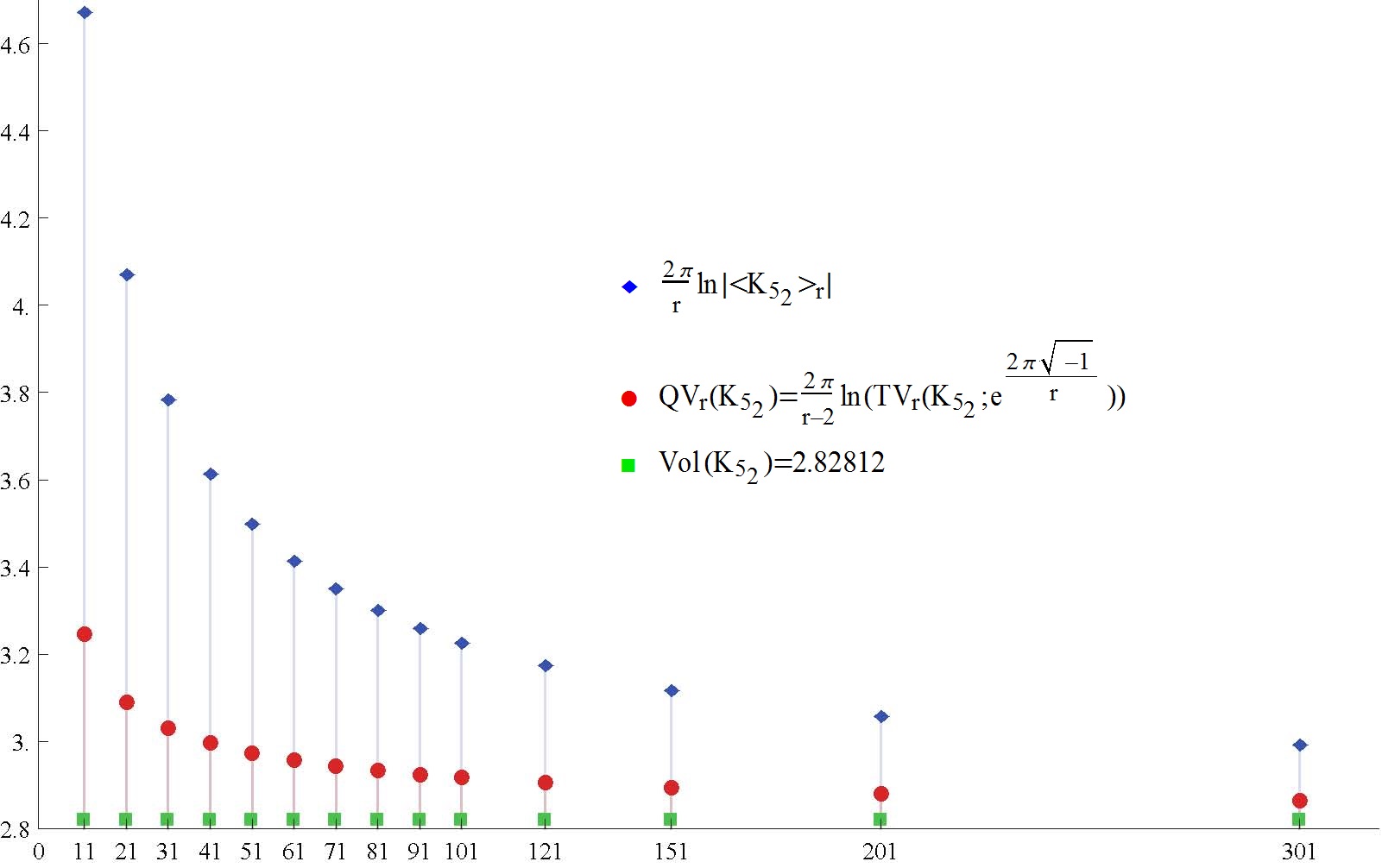}
\caption{Comparison of different invariants for $K_{5_2}$}
\label{Fig2}
\end{figure}

The manifold $M_{3_6}$ in the Callahan-Hildebrand-WeeksÕ census \cite{CHW99} is also the complement of the  $(-2,3,7)-$pretzel knot of Figure~\ref{fig:M36}. It has the same volume as $S^3\setminus K_{5_2}$, namely $$\mathrm{vol}(M_{3_6})=\mathrm{vol}(S^3\setminus K_{5_2})\approx2.82812.
$$

\begin{figure}[htbp]

\includegraphics[width=4cm]{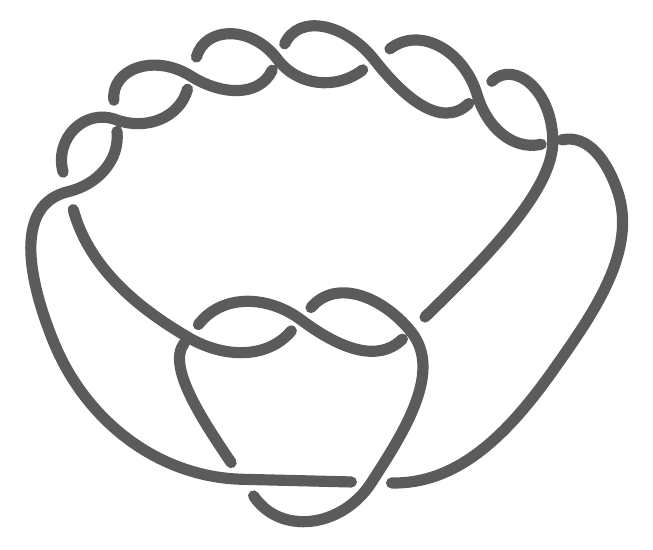}
\caption{The $(-2,3,7)-$pretzel knot}
\label{fig:M36}
\end{figure}

 According to Regina, $M_{3_6}$ can be represented by the ideal triangulation of Figure \ref{fig:M36Triang}. 

\begin{figure}[htbp]
\includegraphics[width=12cm]{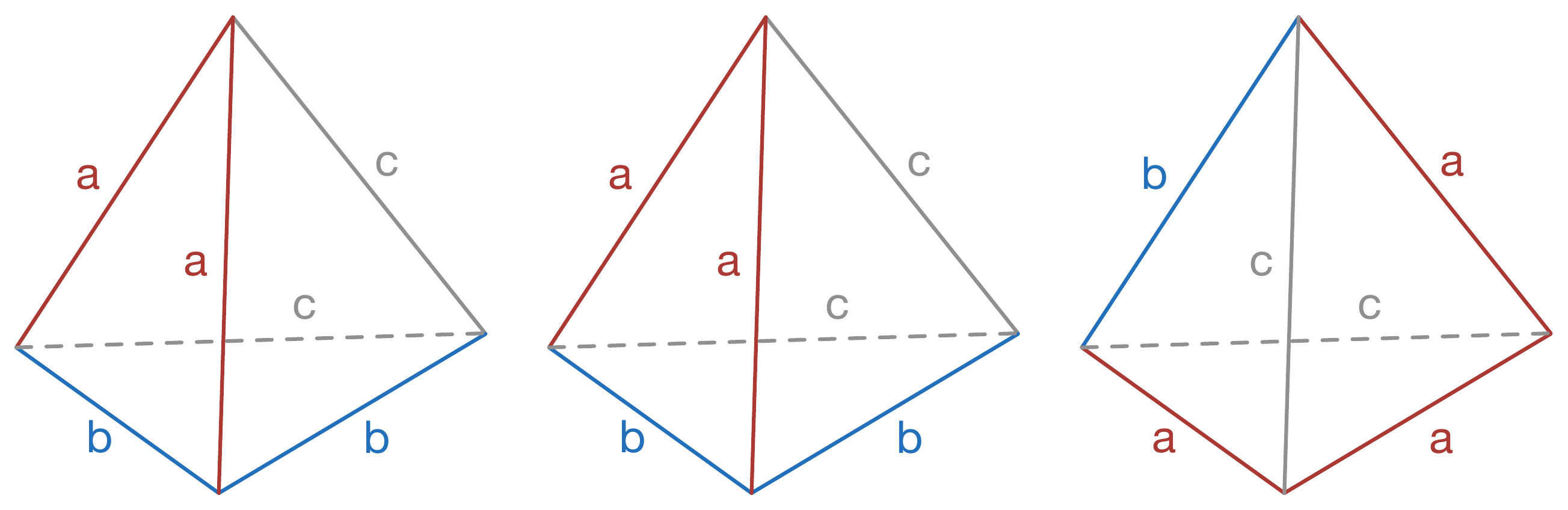}

\caption{}
\label{fig:M36Triang}
\end{figure}

Then for $r\geqslant 3$, we have
\begin{equation*}
\mathrm{TV}_r(M_{3_6})=\sum_{(a,b,c)\in A_r} w_aw_bw_c
\bigg|
      \begin{matrix}
        a & a & b \\
        b & c & c \\
      \end{matrix} \bigg|\bigg|
      \begin{matrix}
        a & a & b \\
        b & c & c \\
      \end{matrix} \bigg|
      \bigg|
      \begin{matrix}
        a & b & c \\
        a & a & c \\
      \end{matrix} \bigg|,
      \end{equation*}
where $A_r$ consists of all triples $(a,b,c)$ of elements of $I_r$ such that $(a,a,b)$, $(b,b,c)$, $(c,c,a)$, $(a,b,c)$ and $(a,a,c)$ are admissible.  

The table below shows a few values of $QV_r(K_{5_2})$ and $QV_r(M_{3_6})$. We observe that  $QV_r(K_{5_2})$ and $QV_r(M_{3_6})$ are distinct, but are getting closer to each other and seem to converge to $2.82812$ as $r$ grows.
 
 \medskip
\centerline{\tiny\begin{tabular}{|c||c|c|c|} 
\hline
 r & $9$ & $11$ & $21$  \\ 
\hline
&&&\\
\ \ $QV_r(K_{5_2})$ \ \ & \ \ $3.3239396087031623282$ \ \  & \ \ $3.2528240712684816477$ \ \  & \ \ $3.0958786489268195966$ \ \    \\
\hline
&&&\\
\ \ $QV_r(M_{3_6})$ \ \ & \ \ $3.2936286562299185780$ \ \  & \ \ $3.2291939333749922011$ \ \  & \ \ $3.0954357480831343159$ \ \    \\
\hline
\end{tabular}}
\smallskip

\centerline{\tiny\begin{tabular}{|c||c|c|c|} 
\hline
 r & $31$ & $51$ & $101$  \\ 
\hline
&&&\\
\ \ $QV_r(K_{5_2})$ \ \ & \ \ $3.0365668215995635907$ \ \  & \ \ $2.9792536251826401549$ \ \  & \ \ $2.9230944207585713174$ \ \    \\
\hline
&&&\\
\ \ $QV_r(M_{3_6})$ \ \ & \ \ $3.0365081953458580040$ \ \  & \ \ $2.9792532229139281449$ \ \  & \ \ $2.9230944207610719723$ \ \    \\
\hline
\end{tabular}}
\bigskip

\medskip

The manifold $M_{3_8}$ in the Callahan-Hildebrand-Weeks census \cite{CHW99}, which is not the complement of any knot in $S^3$, also
has the same volume as $S^3\setminus K_{5_2}$. According to Regina, $M_{3_8}$ has an ideal triangulation  that has the same colors as that of $S^3\setminus K_{5_2}$ drawn above. As a consequence, $QV_r(M_{3_8})$ coincides with $QV_r(K_{5_2})$  for all $r\geqslant 3$.


\subsection{The $K_{6_1}$ knot complement}\label{61}

According to SnapPy \cite{Sna} and Regina \cite{Reg}, the complement of the knot $K_{6_1}$  has volume $$\mathrm{vol}(S^3\setminus K_{6_1})\approx3.16396,$$ and can be described by the ideal triangulation of Figure \ref{fig:K61}.

\begin{figure}[htbp]
\includegraphics[width=15cm]{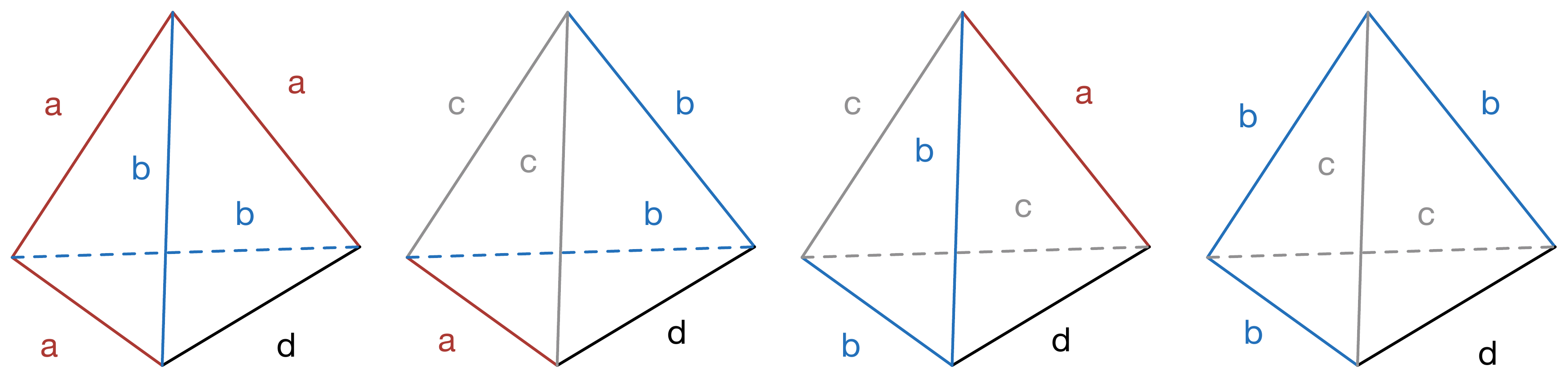}

\caption{}
\label{fig:K61}
\end{figure}

This gives
\begin{equation*}
\mathrm{TV}_r(K_{6_1})=\sum_{(a,b,c,d)\in A_r} w_aw_bw_cw_d
\bigg|
      \begin{matrix}
        a & a & b \\
        a & d & b \\
      \end{matrix} \bigg|\bigg|
      \begin{matrix}
        a & c & c \\
        b & b & d \\
      \end{matrix} \bigg|\bigg|
      \begin{matrix}
        b & b & c \\
        a & c & d \\
      \end{matrix} \bigg|
      \bigg|
      \begin{matrix}
        b & b & c \\
        b & d & c \\
      \end{matrix} \bigg|,
      \end{equation*}
      where $A_r$ consists of quadruples $(a,b,c,d)$ of elements of $I_r$ such that all the triples involved are admissible. From this, we have the following table of values of $QV_r(K_{6_1})$. 
\bigskip

\centerline{\tiny\begin{tabular}{|c||c|c|c|c|c|c|c|c|} 
\hline
 r & $5$ & $7$ & $9$ & $11$ & $21$ & $31$ & $41$   & $51$  \\ 
\hline
&&&&&&&&\\
\ \ $QV_r(K_{6_1})$ \ \ & \ \  $3.83348$  \ \   & \ \ $3.63472$ \ \ &\ \ $3.46573$ \ \ &\ \ $3.39987$ \ \ &\ \  $3.34732$ \ \ & \ \ $3.31699$ \ \ &\ \  $3.29688$ \ \  &\ \  $3.28214$ \ \   \\
\hline
\end{tabular}}
\smallskip

\centerline{\tiny\begin{tabular}{|c||c|c|c|c|c|c|c|c|} 
\hline
$ r$& $61$ & $71$ & $81$ & $91$ & $101$ & $121$ & $151$ & $201$  \\ 
\hline
&&&&&&&&\\
\ \ $QV_r(K_{6_1})$ \ \ & \ \ $3.27076$ \ \ & \ \ $3.26165$ \ \ & \ \ $3.25417$ \ \ & \ \ $ 3.24790$ \ \ & \ \ $3.24255$ \ \  & \ \ $ 3.23390$ \ \ & \ \ $ 3.22431$ \ \ & \ \ $3.21353$  \ \  \\
\hline
\end{tabular}}
\bigskip

Figure \ref{Fig61} compares a few values of  $QV_r(K_{6_1})$ with those of the Kashev invariant $\langle K_{6_1}\rangle_r$ and with the volume  $\mathrm{vol}(S^3\setminus K_{6_1})$.

\begin{figure}[htbp]
\centering
\includegraphics[scale=0.6]{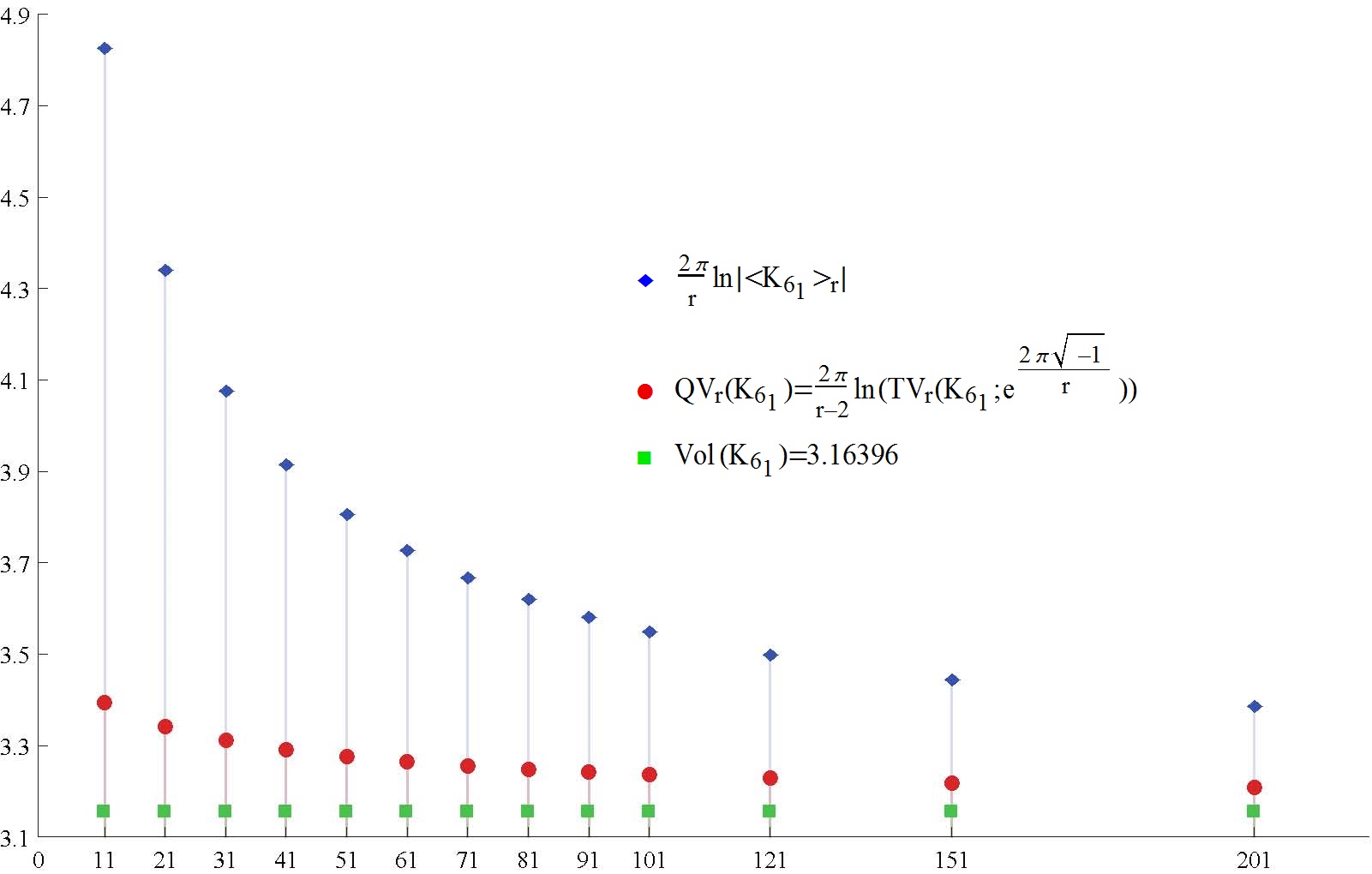}
\caption{Comparison of different invariants for $K_{6_1}$}
\label{Fig61}
\end{figure}
\subsection{Some non-orientable cusped $3$-manifolds}

\subsubsection{The Gieseking manifold} 

The manifold $N_{1_1}$ in the Callahan-Hildebrand-Weeks census \cite{CHW99}, 
also known as the Gieseking manifold, is the smallest non-orientable cusped $3$-manifold. It has an ideal triangulation with a single tetrahedron which, by an Euler characteristic calculation, has only one edge.  According to SnapPy \cite{Sna} and Regina \cite{Reg}, the Gieseking manifold has volume 
$$\mathrm{vol}(N_{1_1})\approx 1.01494.$$
 By Definition \ref{TV}, we have
\begin{equation*}
\mathrm{TV}_r(N_{1_1})=\sum_{a\in A_r} w_a
\bigg|
      \begin{matrix}
        a & a & a \\
        a & a & a \\
      \end{matrix} \bigg|,
      \end{equation*}
      where $A_r$ consist of integers $a$ such that $0\leqslant a\leqslant \lfloor(r-2)/3\rfloor$. Here $\lfloor\  \rfloor$ is the floor function that $\lfloor x\rfloor$ equals the greatest integer less than or equal to $x$. From this, we have the following table of values of $QV_r(N_{1_1})$.
\bigskip

\centerline{\tiny \begin{tabular}{|c||c|c|c|c|c|c|c|c|} 
\hline
 r &   $7$ & $9$ & $11$ & $21$ & $31$ & $41$ & $51$ & $61$   \\ 
\hline
&&&&&&&&\\
\ \ $QV_r(N_{1_1})$ \ \ & \ \  $1.81736$  \ \   & \ \ $1.66782$ \ \ &\ \ $1.62276$ \ \ &\ \ $1.43255$ \ \ &\ \  $1.33012$ \ \ & \ \ $ 1.27064$ \ \ &\ \  $1.23174$ \ \ &\ \  $1.20411$ \ \   \\
\hline
\end{tabular}}
\smallskip

\centerline{\tiny\begin{tabular}{|c||c|c|c|c|c|c|c|c|} 
\hline
$ r$  &$71$ & $81$ & $91$ & $101$ & $201$ & $301$ & $401$  &  $501$     \\ 
\hline
&&&&&&&&\\
\ \ $QV_r(N_{1_1})$ \ \ & \ \  $1.18335$ \ \ & \ \ $ 1.16711$ \ \   &   \ \ $ 1.15401$ \ \ & \ \ $ 1.14319$  \ \  & \ \ $ 1.08943$ \ \ & \ \ $1.06872$ \ \ & \ \ $1.05748$  \ \ & \ \ $1.05035$ \ \  \\
\hline
\end{tabular}}
\bigskip


\subsubsection{ Manifold $N_{2_1}$}
  According to SnapPy \cite{Sna} and Regina \cite{Reg}, the manifold $N_{2_1}$ in Callahan-Hildebrand-Weeks census \cite{CHW99} has volume 
$$\mathrm{vol}(N_{2_1})\approx  1.83193,$$
 and has the following ideal triangulation.
\[\includegraphics[width=8cm]{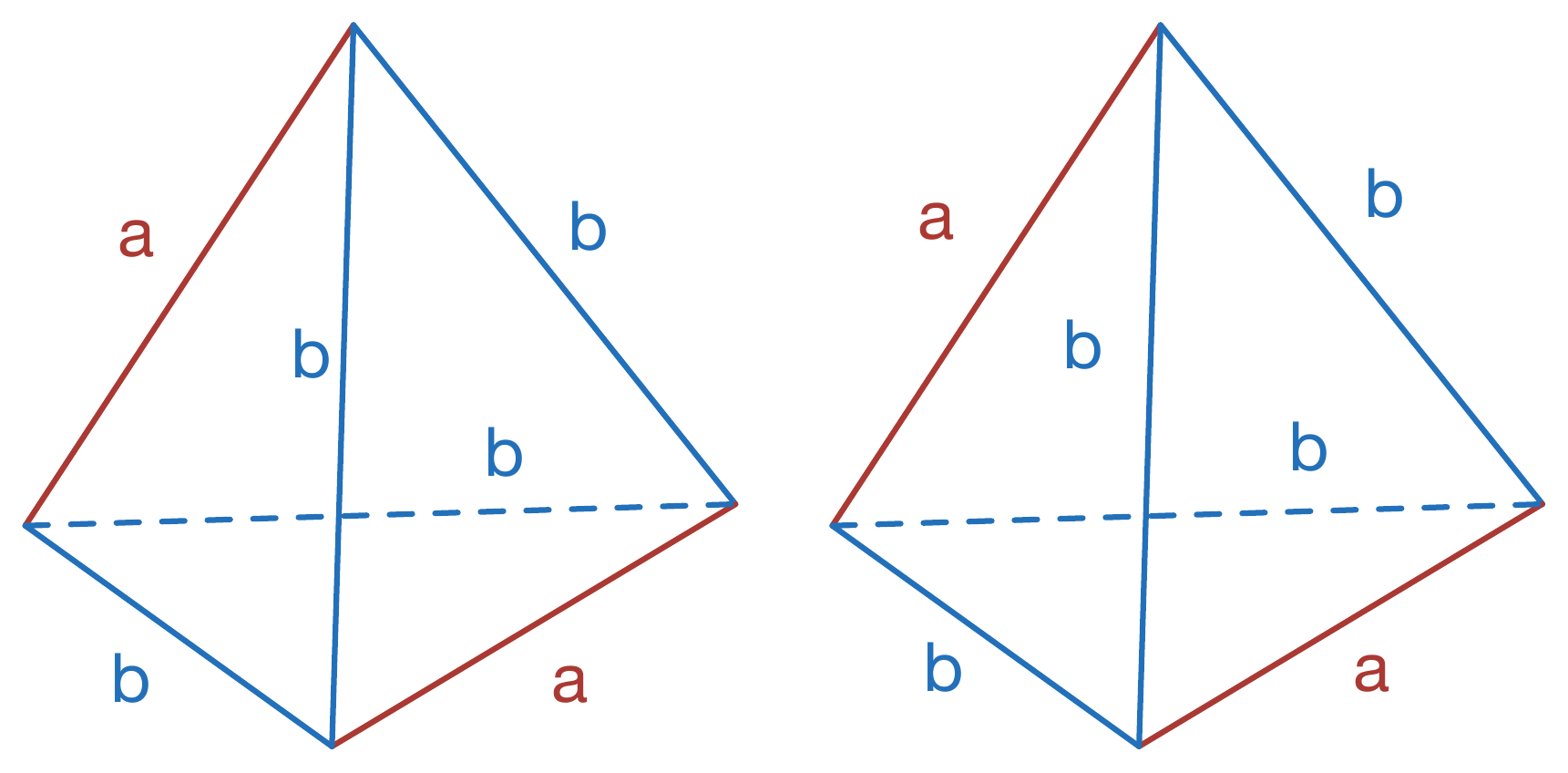}\]

By Definition \ref{TV}, we have 
\begin{equation*}
\mathrm{TV}_r(N_{2_1})=\sum_{(a,b)\in A_r} w_aw_b
\bigg|
      \begin{matrix}
        a & b & b \\
        a & b & b \\
      \end{matrix} \bigg|\bigg|
      \begin{matrix}
        a & b & b \\
        a & b & b \\
      \end{matrix} \bigg|,
      \end{equation*}
      where $A_r$ consist of the pairs $(a,b)$ of elements of $I_r$ such that $(a,b,b)$ is admissible. From this, we have the following table of values of $QV_r(N_{2_1})$.
\bigskip

\centerline{\tiny\begin{tabular}{|c||c|c|c|c|c|c|c|c|} 
\hline
 r & $5$  & $7$ & $9$ & $11$ & $21$ & $31$ & $41$ & $51$ \\ 
\hline
&&&&&&&&\\
\ \ $QV_r(N_{2_1})$  \ \ &\ \  $ 2.90345$ \ \ &\ \  $2.54929$  \ \ & \ \  $2.46119$  \ \   & \ \ $2.42036$ \ \ &\ \ $ 2.20099$ \ \ &\ \ $2.11235$ \ \ &\ \  $2.06163$ \ \ & \ \ $2.02810$ \ \   \\
\hline
\end{tabular}}
\smallskip

\centerline{\tiny\begin{tabular}{|c||c|c|c|c|c|c|c|c|} 
\hline
$ r$ &  $61$  &  $71$   & $81$ & $91$ & $101$  & $121$ &  $151$ & $201$   \\ 
\hline
&&&&&&&&\\
\ \ $QV_r(N_{2_1})$ \ \ &  \ \ $2.00403$ \ \ & \ \ $ 1.98578$  \ \  & \ \  $ 1.97140$ \ \ & \ \ $ 1.95974$ \ \   &   \ \ $1.95006$ \ \  & \ \ $ 1.93489$ \ \ & \ \ $1.91876$ \ \ & \ \ $1.90140$  \ \  \\
\hline
\end{tabular}}
\bigskip


\subsection{Smallest hyperbolic $3$-manifolds with totally geodesic boundary}\label{genus2}

By \cite{KM91}, any orientable hyperbolic $3$-manifold $M_{\text{min}}$ with non-empty totally geodesic boundary that has  minimum volume admits a tetrahedral decomposition by two regular truncated hyperideal tetrahedra of dihedral angles $\pi/6$. As a consequence, such a hyperbolic manifold has volume 
$$\mathrm{vol}(M_{\text{min}})\approx6.452.$$
Such minimums are not unique and are classified in \cite{Fuj90}. In particular, the boundary of each of them is a connected surface of genus $2$, and an Euler characteristic calculation shows that each ideal triangulation of $M_{\text{min}}$ with two tetrahedra has only one edge, as in Figure~\ref{fig:genus2}. 
\begin{figure}[htbp]
\includegraphics[width=8cm]{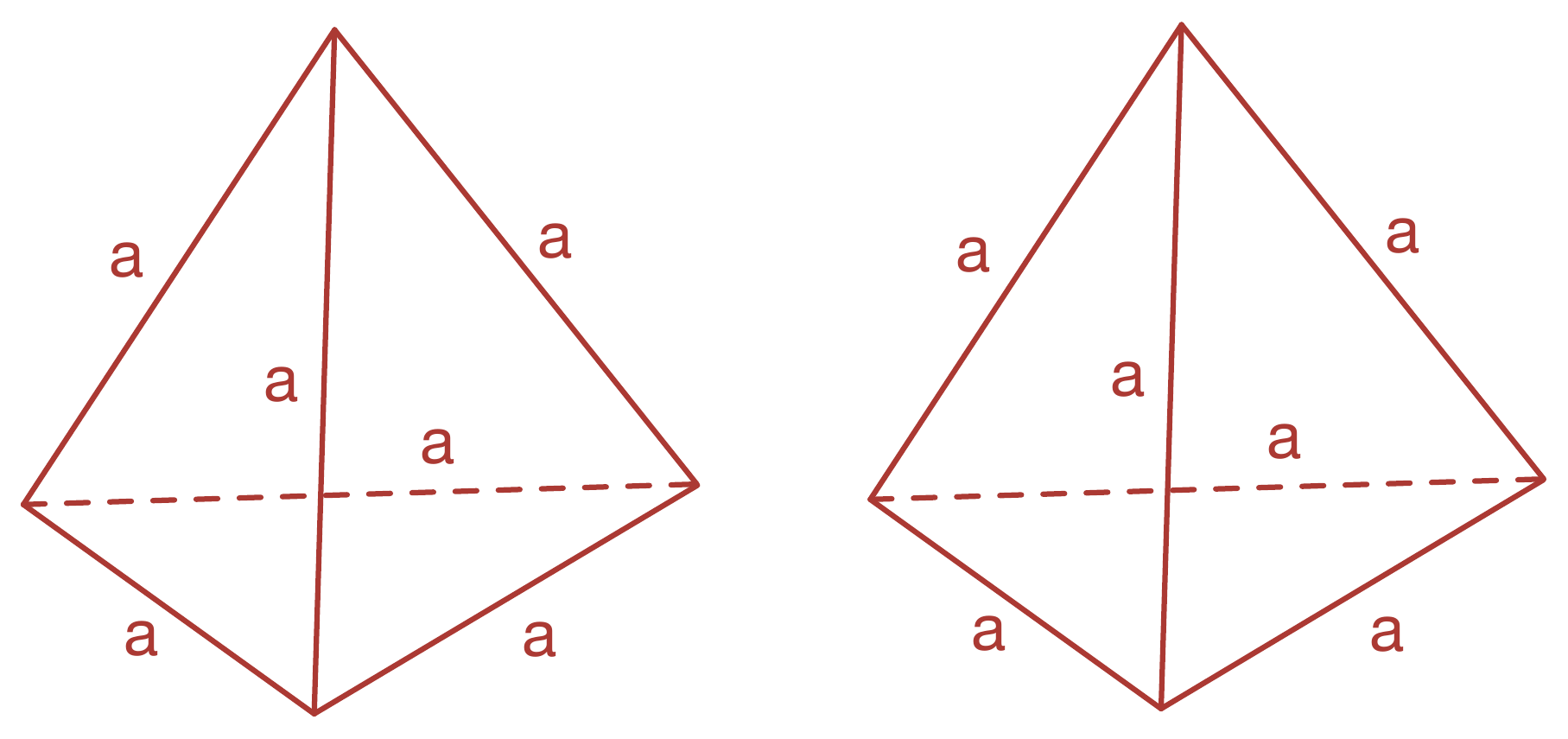}

\caption{}
\label{fig:genus2}
\end{figure}

Therefore, for $r\geqslant 3$, 

\begin{equation*}
\mathrm{TV}_r(M_{\text{min}})=\sum_{a\in A_r} w_a
      \bigg|
      \begin{matrix}
        a & a & a \\
        a & a & a \\
      \end{matrix} \bigg|
           \bigg|
      \begin{matrix}
        a & a & a \\
        a & a & a \\
      \end{matrix} \bigg|,
      \end{equation*}
where $A_r$ consists of all integers $a$ with $0\leqslant a\leqslant \lfloor (r-2)/3\rfloor$. Here $\mathrm{TV}_r(M_{\text{min}})$ is negative for some values of $r$. In this case, we require the argument of the logarithm to be in the interval $[0,2\pi)$, so that the imaginary part of $QV_r(M_{\mathrm{min}}) = \frac{2\pi}{r-2}\log\big(\mathrm{TV}_r(M_{\mathrm{min}})\big)$ is either $0$ or $2\pi^2/(r-2)$. As a consequence, this imaginary part converges to $0$ and it suffices to consider the real part  to test the convergence of $QV_r(M_{\text{min}})$. 

We have the following table of the values of the real part $\Re (QV_r(M_{\text{min}}))$ of $QV_r(M_{\text{min}})$.

\medskip
\centerline{\tiny\begin{tabular}{|c||c|c|c|c|c|c|c|c|} 
\hline
 $r$ & $11$ & $21$ & $31$ & $41$ & $51$ & $61$ & $71$ & $81$ \\ 
\hline
&&&&&&&&\\
\ \ $\Re (QV_r(M_{\text{min}}))$ \ \ &\ \ $4.39782$ \ \ & \ \ $5.12434$\ \ & \ \ $5.44590$ \ \ &
\ \ $5.63235$\ \ & \ \ $5.75566$ \ \ & \ \ $5.84395$ \ \  & \ \ $5.91063$ \ \ & \ \ $5.96297$ \ \ \\
\hline
\end{tabular}}
\smallskip

\centerline{\tiny\begin{tabular}{|c||c|c|c|c|c|c|c|c|} 
\hline
 $r$ & $91$ & $101$ & $201$ & $301$ & $401$ & $501$ & $1001$ & $2001$ \\ 
\hline
&&&&&&&&\\
\ \ $\Re (QV_r(M_{\text{min}}))$ \ \ &\ \ $6.00526$ \ \ & \ \ $6.04022$\ \ & \ \ $6.21400$ \ \ &
\ \ $6.28075$\ \ & \ \ $6.31684$ \ \ & \ \ $6.33970$ \ \ & \ \ $6.38935$ \ \  & \ \ $6.41741$ \ \ \\
\hline
\end{tabular}}
\medskip

Figure \ref{Fig3} below illustrates the asymptotic behavior of $QV_r(M_{\text{min}})$, where the dots represent the points $(r,QV_r(M_{\text{min}}))$ and the squares represent the points $(r, \mathrm{vol}(M_{\text{min}}))$.

\begin{figure}[htbp]
\centering
\includegraphics[scale=0.6]{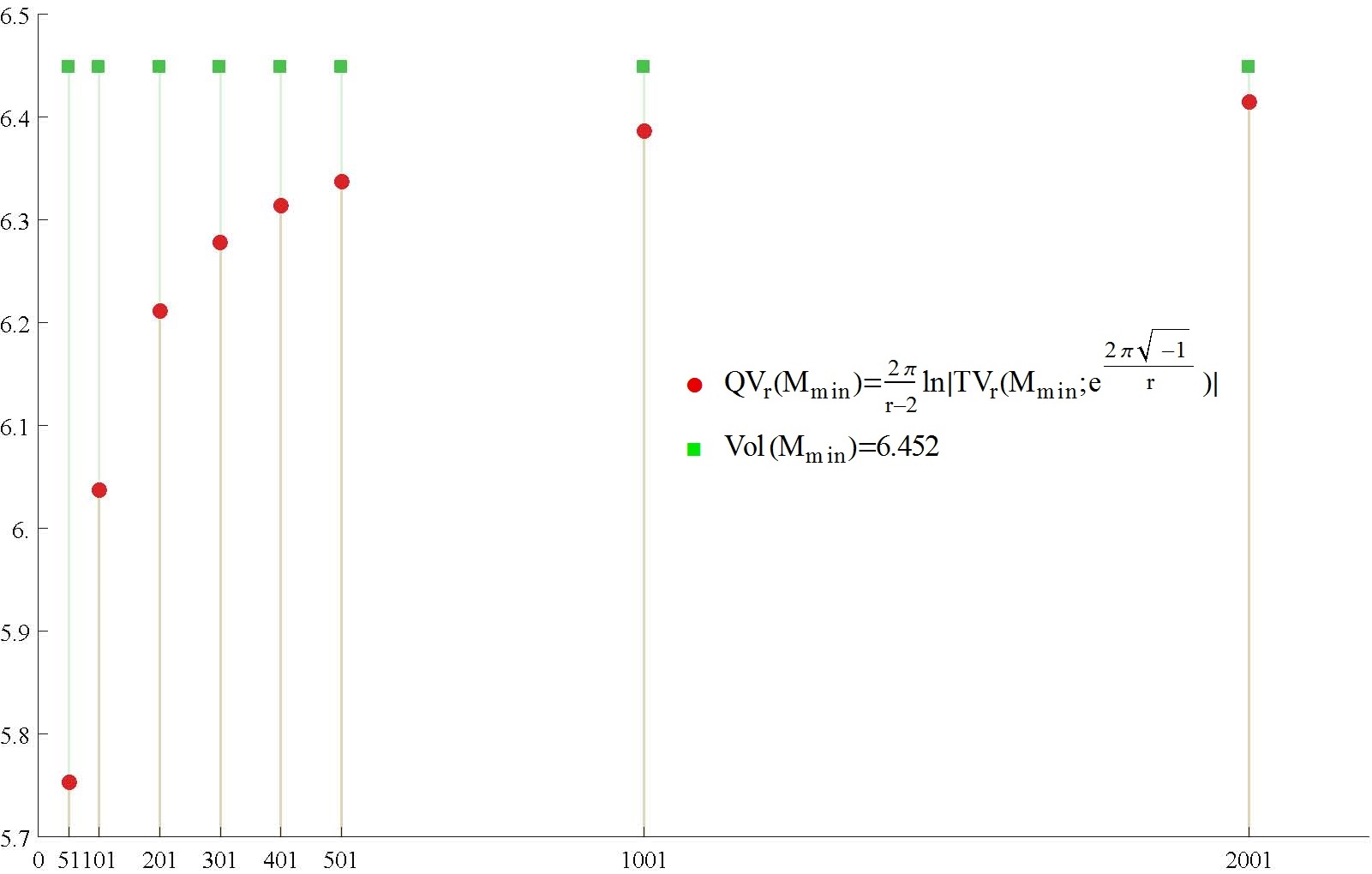}
\caption{Asymptotics of $QV_r(M_{\text{min}})$}
\label{Fig3}
\end{figure}


\section{Evidence for closed $3$-manifolds}\label{closed}

We now consider Conjectures \ref{VC} and \ref{RTVC}   in the case of closed manifolds (\cite{Ro, Tur10, W}), and provide evidence for these conjectures by numerically calculating the Reshetikhin-Turaev invariants 
of a few closed $3$-manifolds obtained by doing Dehn surgeries along the knots $K_{4_1}$ and $K_{5_2}$.

According to \cite{Li}, if $M$ is obtained from $S^3$ by doing a $p$-surgery along a knot $K$, then for an odd $r\geqslant 3$  the Reshetikhin-Turaev invariant $\mathrm{RT}_r(M;q)$ of $M$ at $q=e^{\frac{2\pi\sqrt{-1}}{r}}$ is calculated as 
\begin{equation}\label{RTL}
\mathrm{RT}_r(M;q)=\frac{2}{r}e^{-\epsilon(p)(-\frac{3}{r}-\frac{r+1}{4})\pi\sqrt{-1}}\bigg(\sum_{n=0}^{r-2}\big({\sin\frac{2(n+1)\pi}{r}}\big)^2(-e^{\frac{\pi\sqrt{-1}}{r}})^{p(n^2+2n)}J_{n+1}(K; e^{\frac{4\pi\sqrt{-1}}{r}})\bigg),
\end{equation}
where $\epsilon(p)$ is the sign of $p$, and $J_n(K;e^{\frac{4\pi\sqrt{-1}}{r}})$ is the value of the $n$-th colored Jones polynomial $J_n(K;t)$  of $K$  at $t=e^{\frac{4\pi\sqrt{-1}}{r}}$, normalized in such a way that $J_n(\text{unknot})=1$.

\begin{remark} The conventions in skein theory (\cite{KL94, BHMV, Li}) make use of a variable $A$ that is, either a primitive $2r$-th root of unity for an integer $r$, or a primitive $r$-th root of unity for an odd integer $r$. The  root of unity $q$ in the definition of the Turaev-Viro invariant then corresponds to $A^2$, while the variable $t$ of the colored Jones polynomial corresponds to $A^{4}$. Formula (\ref{RTL})  deals with the case where $A=e^{\frac{\pi\sqrt{-1}}{r}}$ for $r$ odd, in which case $q=e^{\frac{2\pi\sqrt{-1}}{r}}$ and $t=e^{\frac{4\pi\sqrt{-1}}{r}}$.

\end{remark}

\begin{remark} Formula (\ref{RTL}) is directly derived from \cite[\S 4.1]{Li}. In Lickorish's notation and letting $A=e^{\frac{\pi\sqrt{-1}}{r}}$, one has $\mu=\frac{1}{\sqrt{r}}\sin\frac{2\pi}{r}$, $\langle\mu \omega\rangle_{U_-}^\sigma=e^{-\epsilon(p)(-\frac{3}{r}-\frac{r+1}{4})\pi\sqrt{-1}}$, $\langle\mu\omega\rangle_U^{-1}=\frac{2}{\sqrt{r}}\sin\frac{2\pi}{r}$, and for $K_p$ the knot $K$ with framing $p$,
$$\langle\omega\rangle_{K_p}=\sum_{n=0}^{r-2}\Big(\frac{\sin\frac{2(n+1)\pi}{r}}{\sin\frac{2\pi}{r}}\Big)^2(-e^{\frac{\pi\sqrt{-1}}{r}})^{p(n^2+2n)}J_{n+1}(K;e^{\frac{4\pi\sqrt{-1}}{r}}).$$ Multiplying the above terms together, one gets formula (\ref{RTL}).\end{remark}

To calculate the growth rate of $\mathrm{RT}_r(M;q)$ as $r$ approaches  infinity, it is equivalent to calculate the limit  of the following quantity
$$Q_r(M)=2\pi\sqrt{-1} \log\Big(\mathrm{RT}_r(M;e^{\frac{2\pi\sqrt{-1}}{r}})/\mathrm{RT}_{r-2}(M;e^{\frac{2\pi\sqrt{-1}}{r-2}})\Big),$$
where the logarithm $\log$ is chosen so that its imaginary part lies in the interval $(-\pi,\pi).$ 

\subsection{Surgeries along the figure-eight knot} 

In this subsection, we denote by $M_p$ the manifold obtained from $S^3$ by doing a $p$-surgery along the figure-eight knot $K_{4_1}$. Recall from \cite{Thu82} that $M_p$ is hyperbolic if and only if $|p|\geqslant 5$. By  \cite{Mur11},  the $n$-th colored Jones polynomial of $K_{4_1}$ equals 
\begin{equation}\label{CJ41}
J_n(K_{4_1},t)=\sum_{k=0}^{n-1}\prod_{i=1}^k(t^{\frac{n-i}{2}}-t^{-\frac{n-i}{2}})(t^{\frac{n+i}{2}}-t^{-\frac{n+i}{2}}).
\end{equation}

 In the tables below, we list the values of $Q_r(M_p)$ modulo $\pi^2\mathbb Z$ for $p=-6, -5,  5, 6, 7, 8$ and for $r=51$, $101$, $151$, $201$, $301$ and $501$. 

\subsubsection{$p=-6$} According to SnapPy \cite{Sna},
$$\mathrm{CS}(M_{-6})+\mathrm{vol}(M_{-6})\sqrt{-1}=- 1.34092+1.28449\sqrt{-1}\mod\pi^2\mathbb Z,$$
and by (\ref{RTL}) and (\ref{CJ41}), we have
\bigskip

\centerline{\tiny\begin{tabular}{|c||c|c|c|} 
\hline
 r & $51$ & $101$ & $151$  \\ 
\hline
&&&\\
\ \ $Q_r(M_{-6})$ \ \ & \ \ $ - 1.34241+1.22717\sqrt{-1}$ \ \  & \ \ $- 1.32879+1.28425 \sqrt{-1}$ \ \  & \ \ $ - 1.33549+1.28440\sqrt{-1}$ \ \    \\
\hline
\end{tabular}}
\smallskip

\centerline{\tiny\begin{tabular}{|c||c|c|c|} 
\hline
 r & $201$ & $301$ & $501$  \\ 
\hline
&&&\\
\ \ $Q_r(M_{-6})$ \ \ & \ \ $ - 1.33786+1.28443 \sqrt{-1}$ \ \  & \ \ $- 1.33956+1.28446 \sqrt{-1}$ \ \  & \ \ $- 1.34043+1.28448 \sqrt{-1}$ \ \    \\
\hline
\end{tabular}}
\bigskip

\subsubsection{$p=-5$}  According to SnapPy \cite{Sna},
$$\mathrm{CS}(M_{-5})+\mathrm{vol}(M_{-5})\sqrt{-1}= -  1.52067+0.98137\sqrt{-1}\mod\pi^2\mathbb Z,$$
and by (\ref{RTL}) and (\ref{CJ41}), we have 
\bigskip

\centerline{\tiny\begin{tabular}{|c||c|c|c|} 
\hline
 r & $51$ & $101$ & $151$  \\ 
\hline
&&&\\
\ \ $Q_r(M_{-5})$ \ \ & \ \ $ -1.50445+0.87410 \sqrt{-1}$ \ \  & \ \ $ -1.51521+0.98003 \sqrt{-1}$ \ \  & \ \ $- 1.51712+0.98130 \sqrt{-1}$ \ \    \\
\hline
\end{tabular}}
\smallskip

\centerline{\tiny\begin{tabular}{|c||c|c|c|} 
\hline
 r & $201$ & $301$ & $501$  \\ 
\hline
&&&\\
\ \ $Q_r(M_{-5})$ \ \ & \ \ $- 1.51865+ 0.98131 \sqrt{-1}$ \ \  & \ \ $- 1.51977+0.98134 \sqrt{-1}$ \ \  & \ \ $ - 1.52035+0.98136 \sqrt{-1}$ \ \    \\
\hline
\end{tabular}}
\bigskip

\subsubsection{$p=5$} According to SnapPy \cite{Sna},
$$\mathrm{CS}(M_{5})+\mathrm{vol}(M_{5})\sqrt{-1}=1.52067+ 0.98137\sqrt{-1}\mod \pi^2\mathbb Z,$$and by (\ref{RTL}) and (\ref{CJ41}), we have

\bigskip

\centerline{\tiny\begin{tabular}{|c||c|c|c|} 
\hline
 r & $51$ & $101$ & $151$  \\ 
\hline
&&&\\
\ \ $Q_r(M_{5})$ \ \ & \ \ $ 1.50445+0.87410 \sqrt{-1}$ \ \  & \ \ $ 1.51521+0.98003 \sqrt{-1}$ \ \  & \ \ $ 1.51712+0.98130 \sqrt{-1}$ \ \    \\
\hline
\end{tabular}}
\smallskip

\centerline{\tiny\begin{tabular}{|c||c|c|c|} 
\hline
 r & $201$ & $301$ & $501$  \\ 
\hline
&&&\\
\ \ $Q_r(M_{5})$ \ \ & \ \ $1.51865+ 0.98131 \sqrt{-1}$ \ \  & \ \ $1.51977+0.98134 \sqrt{-1}$ \ \  & \ \ $ 1.52035+0.98136 \sqrt{-1}$ \ \    \\
\hline
\end{tabular}}
\bigskip

\subsubsection{$p=6$} According to SnapPy \cite{Sna},
$$\mathrm{CS}(M_{6})+\mathrm{vol}(M_{-6})\sqrt{-1}=1.34092+1.28449\sqrt{-1}\mod \pi^2\mathbb Z,$$
and by (\ref{RTL}) and (\ref{CJ41}), we have
\bigskip

\centerline{\tiny\begin{tabular}{|c||c|c|c|} 
\hline
 r & $51$ & $101$ & $151$  \\ 
\hline
&&&\\
\ \ $Q_r(M_{6})$ \ \ & \ \ $  1.34241+1.22717\sqrt{-1}$ \ \  & \ \ $ 1.32879+1.28425 \sqrt{-1}$ \ \  & \ \ $  1.33549+1.28440\sqrt{-1}$ \ \    \\
\hline
\end{tabular}}
\smallskip

\centerline{\tiny\begin{tabular}{|c||c|c|c|} 
\hline
 r & $201$ & $301$ & $501$  \\ 
\hline
&&&\\
\ \ $Q_r(M_{6})$ \ \ & \ \ $ 1.33786+1.28443 \sqrt{-1}$ \ \  & \ \ $1.33956+1.28446 \sqrt{-1}$ \ \  & \ \ $ 1.34043+1.28448 \sqrt{-1}$ \ \    \\
\hline
\end{tabular}}
\bigskip
 
\subsubsection{$p=7$} According to SnapPy \cite{Sna},
$$\mathrm{CS}(M_{7})+\mathrm{vol}(M_{7})\sqrt{-1}=1.19653+1.46378\sqrt{-1}\mod \pi^2\mathbb Z,$$
and by (\ref{RTL}) and (\ref{CJ41}), we have 
\bigskip
  
\centerline{\tiny\begin{tabular}{|c||c|c|c|} 
\hline
 r & $51$ & $101$ & $151$  \\ 
\hline
&&&\\
\ \ $Q_r(M_{7})$ \ \ & \ \ $1.10084+1.43670 \sqrt{-1}$ \ \  & \ \ $1.18016+1.46354 \sqrt{-1}$ \ \  & \ \ $ 1.18930+1.46367\sqrt{-1}$ \ \    \\
\hline
\end{tabular}}
\smallskip

\centerline{\tiny\begin{tabular}{|c||c|c|c|} 
\hline
 r & $201$ & $301$ & $501$  \\ 
\hline
&&&\\
\ \ $Q_r(M_{7})$ \ \ & \ \ $ 1.19246+ 1.46372 \sqrt{-1}$ \ \  & \ \ $ 1.19472+1.46375 \sqrt{-1}$ \ \  & \ \ $ 1.19588+ 1.46377\sqrt{-1}$ \ \    \\
\hline
\end{tabular}}
\bigskip

\subsubsection{$p=8$} According to SnapPy \cite{Sna},
$$\mathrm{CS}(M_{8})+\mathrm{vol}(M_{8})\sqrt{-1}=1.07850+1.58317\sqrt{-1}\mod \pi^2\mathbb Z,$$
and by (\ref{RTL}) and (\ref{CJ41}), we have 
\bigskip
  
\centerline{\tiny\begin{tabular}{|c||c|c|c|} 
\hline
 r & $51$ & $101$ & $151$  \\ 
\hline
&&&\\
\ \ $Q_r(M_{8})$ \ \ & \ \ $ 0.96311+1.57167 \sqrt{-1}$ \ \  & \ \ $1.05821+1.58282 \sqrt{-1}$ \ \  & \ \ $1.06949+1.58304 \sqrt{-1}$ \ \    \\
\hline
\end{tabular}}
\smallskip

\centerline{\tiny\begin{tabular}{|c||c|c|c|} 
\hline
 r & $201$ & $301$ & $501$  \\ 
\hline
&&&\\
\ \ $Q_r(M_{8})$ \ \ & \ \ $ 1.07343+ 1.58309 \sqrt{-1}$ \ \  & \ \ $1.07625+ 1.58313 \sqrt{-1}$ \ \  & \ \ $1.07769+1.58315 \sqrt{-1}$ \ \    \\
\hline
\end{tabular}}
\bigskip


\subsection{Surgeries along $K_{5_2}$}

In this subsection, we let  $M_p$ be the manifold obtained from $S^3$ by doing a $p$-surgery along the knot $K_{5_2}$. Recall that $M_p$ is hyperbolic if and only if $p\leqslant -1$ or $p\geqslant 5$. By \cite{Mas03}, the $n$-th colored Jones polynomial of $K_{5_2}$ is equal to 
\begin{equation}\label{CJ52}
J_n(K_{5_2},t)=\sum_{k=0}^{n-1}t^{-\frac{k(k+3)}{4}}c_k\prod_{i=1}^k(t^{\frac{n-i}{2}}-t^{-\frac{n-i}{2}})(t^{\frac{n+i}{2}}-t^{-\frac{n+i}{2}}),
\end{equation}
where 
$$c_k=(-1)^kt^{-\frac{5k^2+7k}{4}}\sum_{i=0}^kt^{-\frac{i^2-2i-3ki}{2}}\frac{[k]!}{[i]![k-i]!}.$$
Here the formula differs from that of \cite{Mas03} by replacing $t$ with $t^{-1}$. This comes from the chirality of $K_{5_2}$. Here we stick to the convention that is used in SnapPy \cite{Sna}, which is the mirror image of the one used in \cite{Mas03}.

In the tables below, we list the values of $Q_r(M_p)$ modulo $\pi^2\mathbb Z$ for $p=-3, -2,  -1, 5, 6, 7$ and for $r=51, 75, 101, 125, 151$ and $201$. 

\subsubsection{$p=-3$}  According to SnapPy \cite{Sna},
$$\mathrm{CS}(M_{-3})+\mathrm{vol}(M_{-3})\sqrt{-1}=- 4.45132+2.10310 \sqrt{-1}\mod \pi^2\mathbb Z,$$
 and by (\ref{RTL}) and (\ref{CJ52}), we have
 \bigskip

\centerline{\tiny\begin{tabular}{|c||c|c|c|} 
\hline
 r & $51$ & $75$ & $101$  \\ 
\hline
&&&\\
\ \ $Q_r(M_{-3})$ \ \ & \ \ $ -4.37951+2.10038\sqrt{-1}$ \ \  & \ \ $ - 4.41819+2.10200\sqrt{-1}$ \ \  & \ \ $- 4.43323+  2.10247 \sqrt{-1}$ \ \    \\
\hline
\end{tabular}}
\smallskip

\centerline{\tiny\begin{tabular}{|c||c|c|c|} 
\hline
 r & $125$ & $151$ & $201$  \\ 
\hline
&&&\\
\ \ $Q_r(M_{-3})$ \ \ & \ \ $- 4.43957+  2.10268 \sqrt{-1}$ \ \  & \ \ $- 4.44329+2.10281 \sqrt{-1}$ \ \  & \ \ $ - 4.44681+  2.10293\sqrt{-1}$ \ \    \\
\hline
\end{tabular}}
\bigskip

\subsubsection{$p=-2$} According to SnapPy \cite{Sna},
$$\mathrm{CS}(M_{-2})+\mathrm{vol}(M_{-2})\sqrt{-1}= - 4.63884+1.84359\sqrt{-1}\mod \pi^2\mathbb Z,$$
  and by (\ref{RTL}) and (\ref{CJ52}), we have 
 \bigskip

\centerline{\tiny\begin{tabular}{|c||c|c|c|} 
\hline
 r & $51$ & $75$ & $101$  \\ 
\hline
&&&\\
\ \ $Q_r(M_{-2})$ \ \ & \ \ $- 4.59073+1.84822 \sqrt{-1}$ \ \  & \ \ $ - 4.61357+ 1.84289\sqrt{-1}$ \ \  & \ \ $ -4.62490+  1.84317 \sqrt{-1}$ \ \    \\
\hline
\end{tabular}}
\smallskip

\centerline{\tiny\begin{tabular}{|c||c|c|c|} 
\hline
 r & $125$ & $151$ & $201$  \\ 
\hline
&&&\\
\ \ $Q_r(M_{-2})$ \ \ & \ \ $ -4.62978+1.84331 \sqrt{-1}$ \ \  & \ \ $ - 4.63265+1.84339\sqrt{-1}$ \ \  & \ \ $-4.63536+  1.84348 \sqrt{-1}$ \ \    \\
\hline
\end{tabular}}
\bigskip

\subsubsection{$p=-1$} According to SnapPy \cite{Sna},
$$\mathrm{CS}(M_{-1})+\mathrm{vol}(M_{-1})\sqrt{-1}= -4.86783+1.39851\sqrt{-1}\mod \pi^2\mathbb Z,$$
 and by (\ref{RTL}) and (\ref{CJ52}), we have 
 \bigskip

\centerline{\tiny\begin{tabular}{|c||c|c|c|} 
\hline
 r & $51$ & $75$ & $101$  \\ 
\hline
&&&\\
\ \ $Q_r(M_{-1})$ \ \ & \ \ $ -4.84865+1.40943\sqrt{-1}$ \ \  & \ \ $  - 4.85045+ 1.39808\sqrt{-1}$ \ \  & \ \ $   -4.85817+  1.39817\sqrt{-1}$ \ \    \\
\hline
\end{tabular}}
\smallskip

\centerline{\tiny\begin{tabular}{|c||c|c|c|} 
\hline
 r & $125$ & $151$ & $201$  \\ 
\hline
&&&\\
\ \ $Q_r(M_{-1})$ \ \ & \ \ $  - 4.86157+ 1.39827\sqrt{-1}$ \ \  & \ \ $- 4.86355+ 1.39834 \sqrt{-1}$ \ \  & \ \ $  - 4.86542+  1.39841\sqrt{-1}$ \ \    \\
\hline
\end{tabular}}
\bigskip

\subsubsection{$p=5$} According to SnapPy \cite{Sna},
$$\mathrm{CS}(M_{5})+\mathrm{vol}(M_{5})\sqrt{-1}=  - 1.52067+0.98137\sqrt{-1}\mod \pi^2\mathbb Z,$$
 and by (\ref{RTL}) and (\ref{CJ52}), we have
 \bigskip

\centerline{\tiny\begin{tabular}{|c||c|c|c|} 
\hline
 r & $51$ & $75$ & $101$  \\ 
\hline
&&&\\
\ \ $Q_r(M_{5})$ \ \ & \ \ $ - 1.50445+0.87410\sqrt{-1}$ \ \  & \ \ $ - 1.48899+0.96890 \sqrt{-1}$ \ \  & \ \ $  - 1.51521+  0.98003 \sqrt{-1}$ \ \    \\
\hline
\end{tabular}}
\smallskip

\centerline{\tiny\begin{tabular}{|c||c|c|c|} 
\hline
 r & $125$ & $151$ & $201$  \\ 
\hline
&&&\\
\ \ $Q_r(M_{5})$ \ \ & \ \ $ - 1.51539+ 0.98098\sqrt{-1}$ \ \  & \ \ $- 1.51712+0.98130 \sqrt{-1}$ \ \  & \ \ $ - 1.51865+  0.98131 \sqrt{-1}$ \ \    \\
\hline
\end{tabular}}
\bigskip

\subsubsection{$p=6$} According to SnapPy \cite{Sna},
$$\mathrm{CS}(M_{6})+\mathrm{vol}(M_{6})\sqrt{-1}= -1.51206+ 1.41406 \sqrt{-1}\mod \pi^2\mathbb Z,$$
 and by (\ref{RTL}) and (\ref{CJ52}), we have
 \bigskip

\centerline{\tiny\begin{tabular}{|c||c|c|c|} 
\hline
 r & $51$ & $75$ & $101$  \\ 
\hline
&&&\\
\ \ $Q_r(M_{6})$ \ \ & \ \ $ - 1.46756+1.40044 \sqrt{-1}$ \ \  & \ \ $ - 1.50631+1.41501 \sqrt{-1}$ \ \  & \ \ $  - 1.50836+  1.41339 \sqrt{-1}$ \ \    \\
\hline
\end{tabular}}
\smallskip

\centerline{\tiny\begin{tabular}{|c||c|c|c|} 
\hline
 r & $125$ & $151$ & $201$  \\ 
\hline
&&&\\
\ \ $Q_r(M_{6})$ \ \ & \ \ $ - 1.50968+1.41356 \sqrt{-1}$ \ \  & \ \ $ - 1.51042+1.41372\sqrt{-1}$ \ \  & \ \ $  -1.51113+  1.41386\sqrt{-1}$ \ \    \\
\hline
\end{tabular}}
\bigskip

\subsubsection{p=7}  According to SnapPy \cite{Sna},
$$\mathrm{CS}(M_{7})+\mathrm{vol}(M_{7})\sqrt{-1}=  - 1.55255+ 1.75713 \sqrt{-1}\mod \pi^2\mathbb Z,$$
 and by (\ref{RTL}) and (\ref{CJ52}), we have
 \bigskip

\centerline{\tiny\begin{tabular}{|c||c|c|c|} 
\hline
 r & $51$ & $75$ & $101$  \\ 
\hline
&&&\\
\ \ $Q_r(M_{7})$ \ \ & \ \ $ - 1.53822+1.75178 \sqrt{-1}$ \ \  & \ \ $ - 1.55297+ 1.75315\sqrt{-1}$ \ \  & \ \ $  - 1.55265+  1.75507 \sqrt{-1}$ \ \    \\
\hline
\end{tabular}}
\smallskip

\centerline{\tiny\begin{tabular}{|c||c|c|c|} 
\hline
 r & $125$ & $151$ & $201$  \\ 
\hline
&&&\\
\ \ $Q_r(M_{7})$ \ \ & \ \ $-1.55257+ 1.75582 \sqrt{-1}$ \ \  & \ \ $ - 1.55255+1.75625 \sqrt{-1}$ \ \  & \ \ $ - 1.55254+  1.75664 \sqrt{-1}$ \ \    \\
\hline
\end{tabular}}
\bigskip

\section{An integrality conjecture for torus link complements}\label{unknot}

In this section, we study the Turaev-Viro invariants for torus link complements. We propose the following Integrality Conjecture \ref{int}, and provide evidence by both rigorous (\S \ref{rigorous}) and numerical (\S \ref{integrality}) calculations. 

\begin{conjecture}\label{int}
Let $T_{(m,n)}$ be the $(m,n)$-torus link in $S^3$. If $r$ is relatively prime to $m$ and $n,$ then $\mathrm{TV}_r(S^3\setminus T_{(m,n)})$ is an integer independent of the choice of the roots of unity $q$.
\end{conjecture} 

\subsection{Calculations for some torus links}\label{rigorous}

In this subsection, we will rigorously calculate $\mathrm{TV}_r(M)$ for the complements of the unknot, the trefoil knot, the Hopf link and the torus links $T_{(2,4)}$ and $T_{(2,6)}$. As in the previous sections,  for a link $L$ in $S^3$ we let $$\mathrm{TV}_r(L)=\mathrm{TV}_r(S^3\setminus L).$$
  All the ideal triangulations used in this section are obtained by using Regina \cite{Reg} and SnapPy \cite{Sna}, and for simplicity, we will omit the arrows on the edges and keep only the colors.

\subsubsection{The unknot} 

\begin{proposition}\label{=1} Let $U$ be the unknot in $S^3$. Then   $$\mathrm{TV}_r(U)=1$$
for all $r\geqslant 3$ and for all $q\in\mathbb C$ such that $q^2$ is a primitive root of unity of degree $r$.
\end{proposition}

\begin{proof}
The complement of the  unknot  admits the  ideal triangulation represented in Figure~\ref{fig:Unknot}.

\begin{figure}[htbp]

\includegraphics[width=8cm]{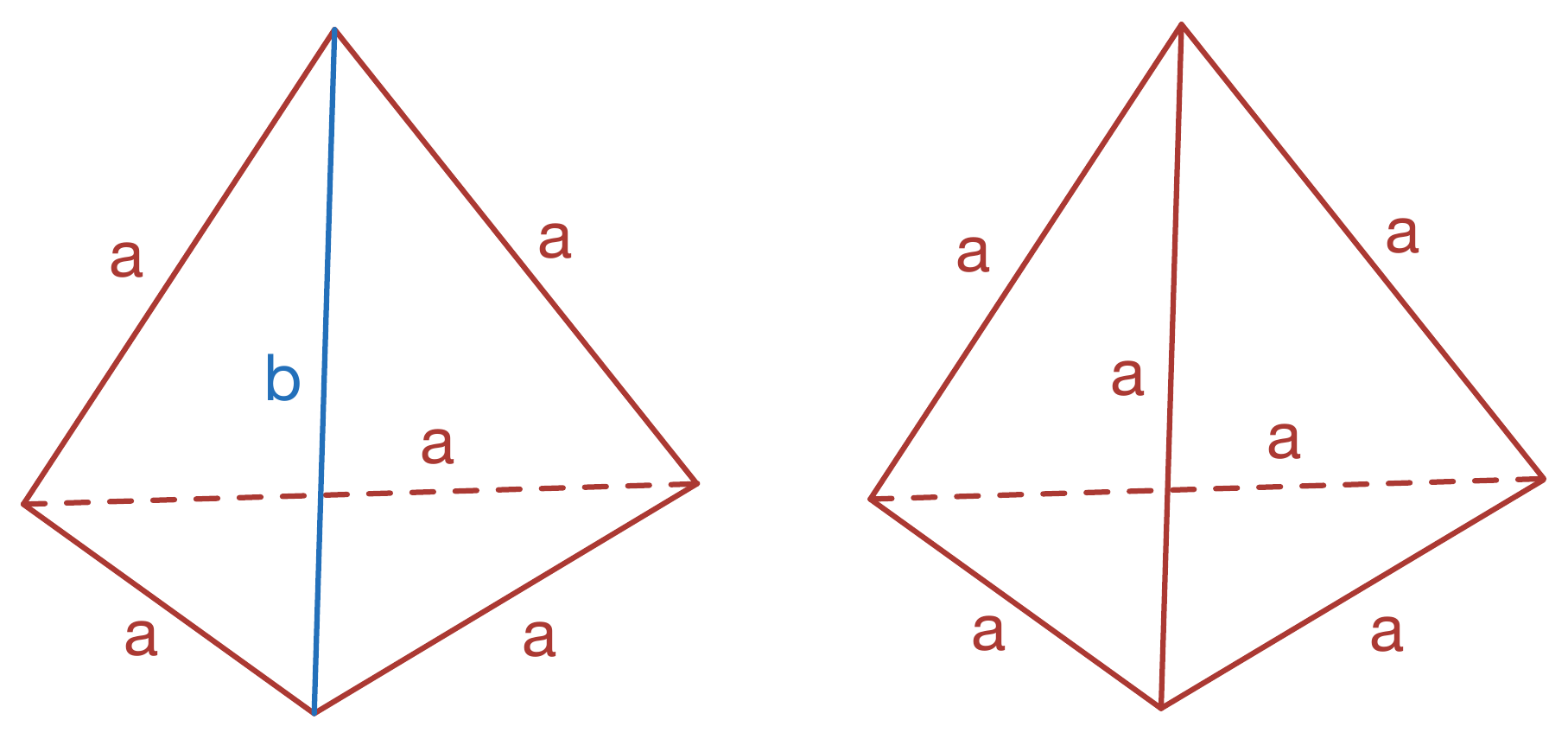}
\caption{}
\label{fig:Unknot}
\end{figure}

Therefore, for each $r\geqslant 3,$ we have
\begin{equation*}
\begin{split}
\mathrm{TV}_r(U)=&\sum_{a,b}w_aw_b\bigg|
      \begin{matrix}
        a & a & a \\
        a & a & b \\
      \end{matrix} \bigg|\bigg|
      \begin{matrix}
        a & a & a \\
        a & a & a \\
      \end{matrix} \bigg|\\
      =&\sum_{a}w_a\bigg|
      \begin{matrix}
        a & a & a \\
        a & a & a \\
      \end{matrix} \bigg|\bigg(\sum_bw_b\bigg|
      \begin{matrix}
        a & a & a \\
        a & a & b \\
      \end{matrix} \bigg|\bigg),
  \end{split}
 \end{equation*}
where in the first row $(a,b)\in I_r\times I_r$ runs over all the admissible colorings at level $r,$ and in the second row $a$ is over all the elements of $I_r$ such that $(a,a,a)$ is admissible and $b$ is over all elements of $I_r$ such that $(a,a,b)$ is admissible. Then the result follows from the following identity
\begin{equation}\label{5.2}
\sum_bw_b\bigg|
      \begin{matrix}
        a & a & a \\
        a & a & b \\
      \end{matrix} \bigg|=\delta_{0,a}.
\end{equation}
To prove (\ref{5.2}), we use the Orthogonality Property. Letting $m=0,$ $s=b$ and $i=j=k=l=n=a$ in  (\ref{O}), we have 
$$\sum_bw_bw_0\bigg|
      \begin{matrix}
        a & a & 0 \\
        a & a & b \\
      \end{matrix} \bigg|\bigg|
      \begin{matrix}
        a & a & a \\
        a & a & b \\
      \end{matrix} \bigg|=\delta_{0,a},$$
where $b$ is over all elements of $I_r$ such that $(a,b)$ is admissible at level $r$. Since $w_0=1$ and  $\bigg|
      \begin{matrix}
        a & a & 0 \\
        a & a & b \\
      \end{matrix} \bigg|=\frac{1}{[2a+1]},$ we have 
      $$\sum_bw_b\bigg|
      \begin{matrix}
        a & a & a \\
        a & a & b \\
      \end{matrix} \bigg|=[2a+1]\cdot\delta_{0,a}=\delta_{0,a}.$$
\end{proof}

\begin{conjecture}\label{detect} Let $K$ be a knot in $S^3$. Then $\mathrm{TV}_r(K)=1$ for all $r\geqslant 3$ and for all $q\in\mathbb C$ such that $q^2$ is a primitive root of unity of degree $r$  if and only if $K=U$. 
\end{conjecture}

\begin{remark} It is interesting to know whether there is an $M\neq S^3\setminus U,$ not necessarily a link complement, such that $\mathrm{TV}_r(M)=1$ for all $r$ and $q$.
\end{remark}

\subsubsection{The trefoil knot}\label{T(2,3)}

\begin{proposition}\label{3_1} Let $T_{(2,3)}$ be the trefoil knot in $S^3$. Then   $$\mathrm{TV}_r( T_{(2,3)})=\lfloor\frac{r-2}{3}\rfloor+1$$
for all $r\geqslant 3$ and for all $q\in\mathbb C$ such that $q^2$ is a primitive root of unity of degree $r$.
\end{proposition}

\begin{proof}
The complement of trefoil knot $T_{(2,3)}$  admits the  ideal triangulation represented in Figure \ref{fig:Trefoil}.

\begin{figure}[htbp]

\includegraphics[width=8cm]{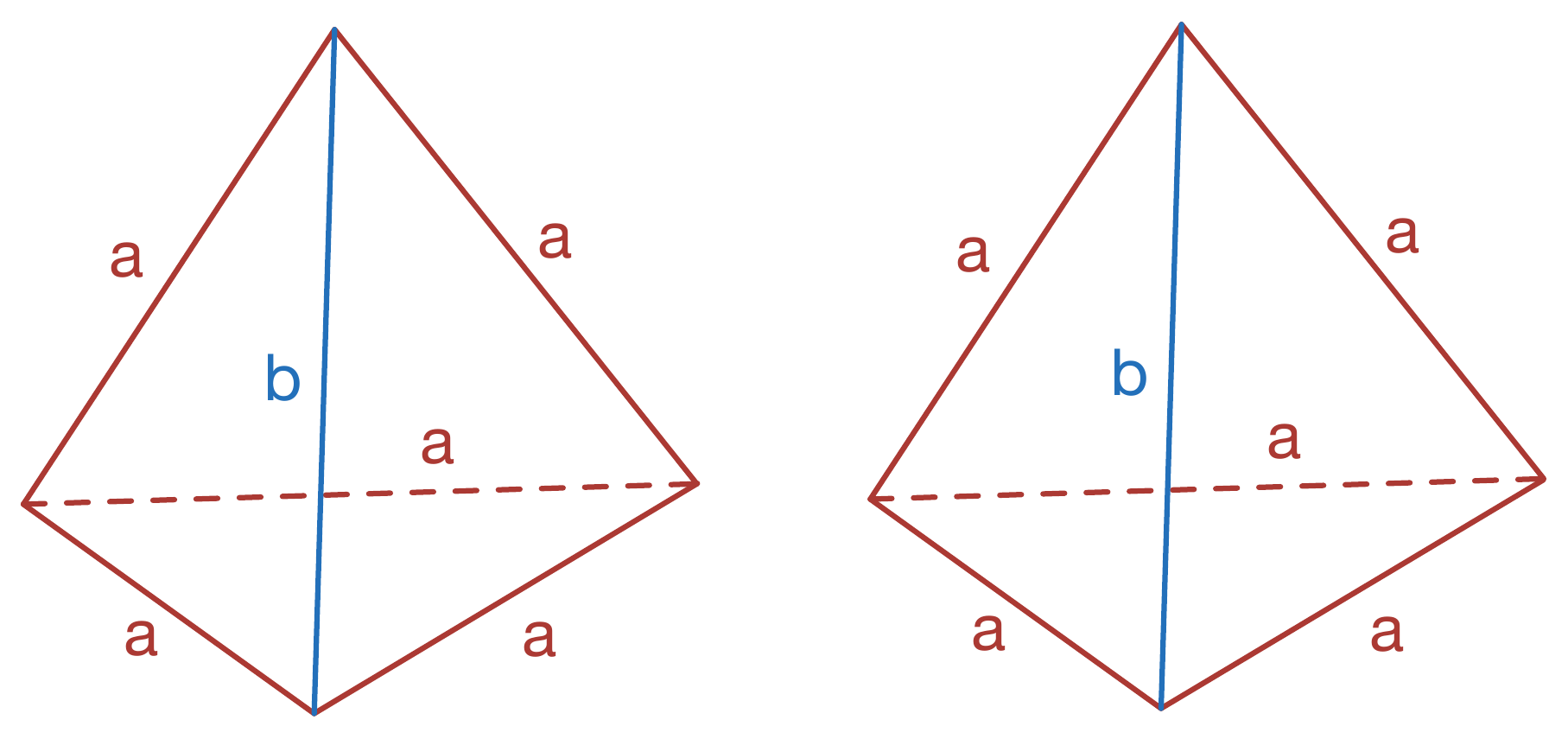}

\caption{}
\label{fig:Trefoil}
\end{figure}

Therefore, for each $r\geqslant 3,$ we have
\begin{equation*}
\mathrm{TV}_r(T_{(2,3)})=\sum_{a,b}w_aw_b\bigg|
      \begin{matrix}
        a & a & a \\
        a & a & b \\
      \end{matrix} \bigg|\bigg|
      \begin{matrix}
        a & a & a \\
        a & a &  b \\
      \end{matrix} \bigg|,
 \end{equation*}
where $(a,b)\in I_r\times I_r$ runs over all the admissible colorings at level $r$. The triple $(a,a,a)$ being admissible implies that $a\in\mathbb Z$ and $a\leqslant (r-2)/3$. Hence the right hand side equals 
$$\sum_{0\leqslant a\leqslant  \frac{r-2}{3}}\bigg(\sum_bw_bw_a\bigg|
      \begin{matrix}
        a & a & a \\
        a & a & b \\
      \end{matrix} \bigg|\bigg|
      \begin{matrix}
        a & a & a \\
        a & a & b \\
      \end{matrix} \bigg|\bigg),$$
where $a$ is over all the  integers in that range and  $b$ is over all elements of $I_r$ such that $(a,a,b)$ is admissible. Letting $i=j=k=l=m=n=a$ and $s=b$ in the Orthogonality Property (\ref{O}), we have 
$$\sum_bw_bw_a\bigg|
      \begin{matrix}
        a & a & a \\
        a & a & b \\
      \end{matrix} \bigg|\bigg|
      \begin{matrix}
        a & a & a \\
        a & a & b \\
      \end{matrix} \bigg|=1,$$
where $b$ is over all elements of $I_r$ such that $(a,a,b)$ is admissible. As a consequence,
$$\mathrm{TV}_r(T_{(2,3)})=\sum_{0\leqslant a\leqslant \frac{r-2}{3}}1=\lfloor \frac{r-2}{3}\rfloor+1.$$ 
\end{proof}

\subsubsection{The Hopf link and torus links $T_{(2,4)}$ and $T_{(2,6)}$}

\begin{proposition}\label{Hopf} Let $T_{(2,2)}$ be the Hopf link in $S^3$. Then   $$\mathrm{TV}_r(T_{(2,2)})=r-1$$
for all $r\geqslant 3$ and for all $q\in\mathbb C$ such that $q^2$ is a primitive root of unity of degree $r$.
\end{proposition}

\begin{proof}
The complement of the Hopf link  admits the  ideal triangulation represented in Figure \ref{fig:Hopf}.

\begin{figure}[htbp]

\includegraphics[width=12cm]{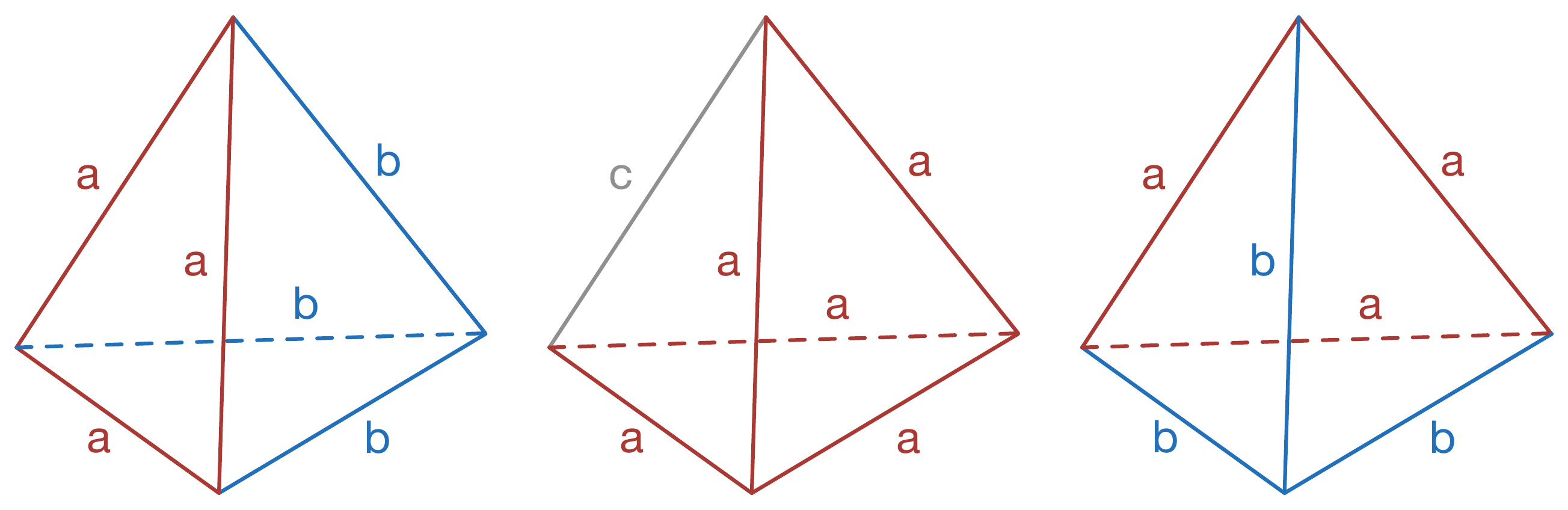}

\caption{}
\label{fig:Hopf}
\end{figure}
Therefore, for each $r\geqslant 3,$ we have
\begin{equation*}\label{5.4}
\begin{split}
\mathrm{TV}_r(T_{(2,2)})=&\sum_{a,b,c}w_aw_bw_c\bigg|
      \begin{matrix}
        a & a & a \\
        a & a & c \\
      \end{matrix} \bigg|\bigg|
      \begin{matrix}
        a & a & a \\
        b & b & b \\
      \end{matrix} \bigg|\bigg|
      \begin{matrix}
        a & a & a \\
        b & b & b \\
      \end{matrix} \bigg|\\
 =& \sum_{a,b}w_aw_b\bigg|
      \begin{matrix}
        a & a & a \\
        b & b & b \\
      \end{matrix} \bigg|\bigg|
      \begin{matrix}
        a & a & a \\
        b & b & b \\
      \end{matrix} \bigg|
       \bigg(\sum_cw_c\bigg|
      \begin{matrix}
        a & a & a \\
        a & a & c \\
      \end{matrix} \bigg|\bigg),
      \end{split}
 \end{equation*}
where in the first row $(a,b,c)$ runs over all the admissible colorings at level $r,$ and in the second row $c$ runs over all elements of $I_r$ such that all the involved quantum $6j$-symbols are admissible. By (\ref{5.2}), we have \begin{equation*}
\sum_cw_c\bigg|
      \begin{matrix}
        a & a & a \\
        a & a & c \\
      \end{matrix} \bigg|=\delta_{0,a}.
\end{equation*}
Therefore, $$\mathrm{TV}_r(T_{(2,2)})=\sum_{b}w_0w_b\bigg|
      \begin{matrix}
        0 & 0 & 0 \\
        b & b & b \\
      \end{matrix}\bigg|\bigg|
      \begin{matrix}
        0 & 0 & 0 \\
        b & b & b \\
      \end{matrix}\bigg|=\sum_{b}1,$$
where $b$ is over all the elements in $I_r$ such that $(0,b,b)$ is admissible. Since this holds for all elements $b$ in $I_r,$ $$\mathrm{TV}_r(T_{(2,2)})=|I_r|=r-1.$$
\end{proof}

\begin{proposition} \label{T(2,4)} Let $T_{(2,4)}$ be the $(2,4)$-torus link in $S^3$. Then   $$\mathrm{TV}_r(T_{(2,4)})=\Big(\lfloor\frac{r-2}{2}\rfloor+1\Big)\Big(\lfloor\frac{r-1}{2}\rfloor+1\Big)$$
for all $r\geqslant 3$ and for all $q\in\mathbb C$ such that $q^2$ is a primitive root of unity of degree $r$.
\end{proposition}

\begin{proof}
The complement of the torus link $T_{(2,4)}$ has the following ideal triangulation represented in Figure~\ref{fig:T24}.

\begin{figure}[htbp]
\includegraphics[width=14cm]{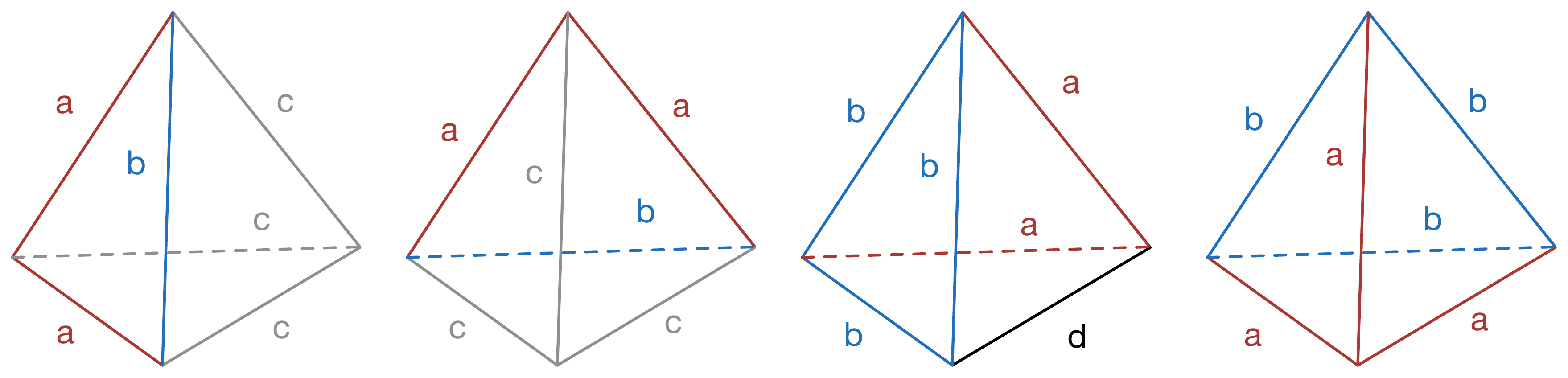}

\caption{}
\label{fig:T24}
\end{figure}

Therefore, for each $r\geqslant 3,$ we have
\begin{equation*}
\begin{split}
\mathrm{TV}_r(T_{(2,4)})=&\sum_{(a,b,c,d)\in A_r}w_aw_bw_cw_d\bigg|
      \begin{matrix}
        a & a & b \\
        c & c & c \\
      \end{matrix} \bigg|\bigg|
      \begin{matrix}
        a & a & b \\
        c & c & c \\
      \end{matrix} \bigg|\bigg|
      \begin{matrix}
        b & b & b \\
        a & a & d \\
      \end{matrix} \bigg|\bigg|
      \begin{matrix}
        b & b & b \\
        a & a & a \\
      \end{matrix} \bigg|\\
=&      \sum_{a,b,c}w_aw_c\bigg|
      \begin{matrix}
        a & a & b \\
        c & c & c \\
      \end{matrix} \bigg|\bigg|
      \begin{matrix}
        a & a & b \\
        c & c & c \\
      \end{matrix} \bigg|\bigg|
      \begin{matrix}
        b & b & b \\
        a & a & a \\
      \end{matrix} \bigg|\bigg(\sum_dw_dw_b\bigg|
      \begin{matrix}
        b & b & b \\
        a & a & d \\
      \end{matrix} \bigg|\bigg),
\end{split}      
\end{equation*}
where in the second row $a,b,c$ run over elements of $I_r$ such that all the involved triples are admissible. We claim that
$$\sum_dw_dw_b\bigg|
      \begin{matrix}
        b & b & b \\
        a & a & d \\
      \end{matrix} \bigg|=\sqrt{-1}^{2a+2b}\sqrt{[2a+1][2b+1]}\cdot \delta_{0,b}.$$
Indeed, letting $m=0,$ $i=j=a,$ $k=l=n=b$ and $s=d$ in the Orthogonality Property (\ref{O}), we have 
$$\sum_dw_dw_b\bigg|
      \begin{matrix}
        b & b & b \\
        a & a & 0 \\
      \end{matrix} \bigg|\bigg|
      \begin{matrix}
        b & b & b \\
        a & a & d \\
      \end{matrix} \bigg|=\delta_{0,b}.$$Then the claim follows from the fact that 
      $$\bigg|
      \begin{matrix}
        b & b & b \\
        a & a & 0 \\
      \end{matrix} \bigg|=\frac{\sqrt{-1}^{2a+2b}}{\sqrt{[2a+1][2b+1]}}.$$
  Therefore,
\begin{equation*}
\begin{split}
\mathrm{TV}_r(T_{(2,4)})=& \sum_{a,c}w_aw_c\bigg|
      \begin{matrix}
        a & a & 0 \\
        c & c & c \\
      \end{matrix} \bigg|\bigg|
      \begin{matrix}
        a & a & 0 \\
        c & c & c \\
      \end{matrix} \bigg|\bigg|
      \begin{matrix}
        0 & 0 & 0 \\
        a & a & a \\
      \end{matrix} \bigg|\sqrt{-1}^{2a}\sqrt{[2a+1]}\\
      =& \sum_{a,c}(-1)^{2a}[2a+1](-1)^{2c}[2c+1]\frac{(-1)^{2a+2c}}{[2a+1][2c+1]}\frac{\sqrt{-1}^{2a}}{\sqrt{[2a+1]}}\sqrt{-1}^{2a}\sqrt{[2a+1]}
      \\
      =&\sum_{a,c}1,
  \end{split}
   \end{equation*}
   where $a,c$ run over all the elements of $I_r$ such that $(c,c,a)$ and $(a,a,0)$ are admissible. Counting the number of such pairs $(a,c),$ we have 
   $$\mathrm{TV}_r(T_{(2,4)})=\Big(\lfloor\frac{r-2}{2}\rfloor+1\Big)\Big(\lfloor\frac{r-1}{2}\rfloor+1\Big).$$
\end{proof}

\begin{proposition} \label{T(2,6)} Let $T_{(2,6)}$ be the $(2,6)$-torus link in $S^3$. Then   $$\mathrm{TV}_r(S^3\setminus T_{(2,6)})=\Big(\lfloor\frac{r-2}{3}\rfloor+1\Big)\Big(\lfloor\frac{2r-2}{3}\rfloor+1\Big)$$
for all $r\geqslant 3$ and for all $q\in\mathbb C$ such that $q^2$ is a primitive root of unity of degree $r$.
\end{proposition}

\begin{proof}
The complement of the torus link $T_{(2,6)}$ has the following ideal triangulation represented in Figure~\ref{fig:T26}.

\begin{figure}[htbp]

\includegraphics[width=14cm]{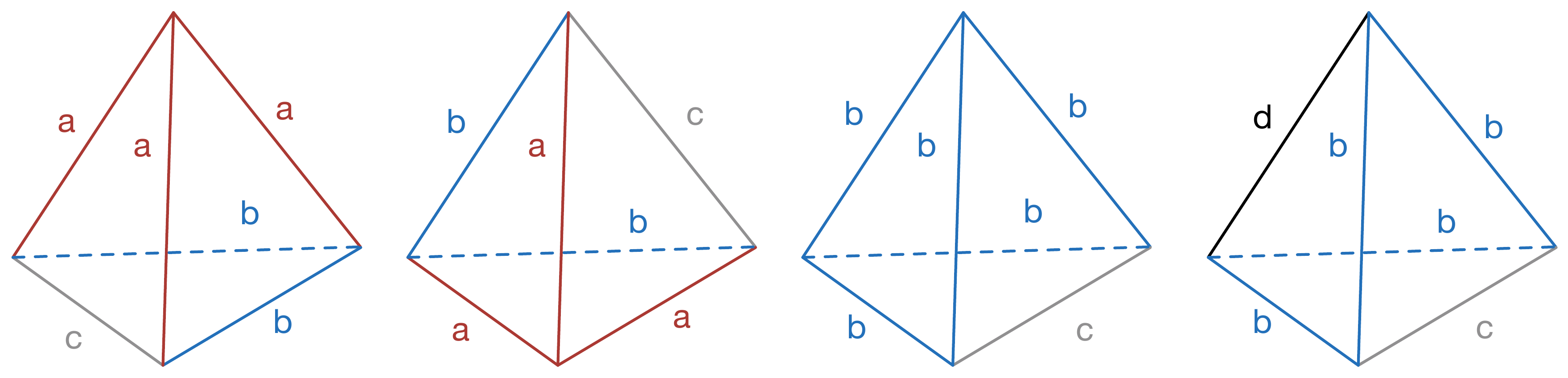}

\caption{}
\label{fig:T26}
\end{figure}

Therefore, for each $r\geqslant 3,$ we have
\begin{equation*}
\begin{split}
\mathrm{TV}_r(T_{(2,6)})=&\sum_{(a,b,c,d)\in A_r}w_aw_bw_cw_d|\bigg|
      \begin{matrix}
        a & a & c \\
        b & b & a \\
      \end{matrix} \bigg|\bigg|
      \begin{matrix}
        a & a & c \\
        b & b & a \\
      \end{matrix} \bigg|\bigg|
      \begin{matrix}
        b & b & c \\
        b & b & b \\
      \end{matrix} \bigg|\bigg|
      \begin{matrix}
        b & b & d \\
        b & b & c \\
      \end{matrix} \bigg|\\
=&      \sum_{a,b,c}w_aw_b\bigg|
      \begin{matrix}
        a & a & c \\
        b & b & a \\
      \end{matrix} \bigg|\bigg|
      \begin{matrix}
        a & a & c \\
        b & b & a \\
      \end{matrix} \bigg|\bigg|
      \begin{matrix}
        b & b & c \\
        b & b & b \\
      \end{matrix} \bigg|\bigg(\sum_dw_dw_c\bigg|
      \begin{matrix}
        b & b & d \\
        b & b & c \\
      \end{matrix} \bigg|\bigg),
\end{split}      
\end{equation*}
where in the second row $a,b,c$ run over elements of $I_r$ such that all the involved triples are admissible. We claim that
$$\sum_dw_dw_c\bigg|
      \begin{matrix}
        b & b & d \\
        b & b & c \\
      \end{matrix} \bigg|=(-1)^{2b}[2b+1]\cdot\delta_{0,c}.$$
Indeed,  letting $m=0,$ $i=j=k=l=b,$ $n=c$ and $s=d$ in the Orthogonality Property (\ref{O}), we have$$\sum_dw_dw_c\bigg|
      \begin{matrix}
        b & b & d \\
        b & b & 0 \\
      \end{matrix} \bigg|\bigg|
      \begin{matrix}
        b & b & d \\
        b & b & c \\
      \end{matrix} \bigg|=\delta_{0,c}.$$
Then the claim follows from the fact that 
      $$\bigg|
      \begin{matrix}
        b & b & d \\
        b & b & 0 \\
      \end{matrix} \bigg|=\frac{(-1)^{2b}}{[2b+1]}.$$
      Therefore,
\begin{equation*}
\begin{split}
\mathrm{TV}_r(T_{(2,6)})=& \sum_{a,b}w_aw_b\bigg|
      \begin{matrix}
        a & a & 0 \\
        b & b & a \\
      \end{matrix} \bigg|\bigg|
      \begin{matrix}
        a & a & 0 \\
        b & b & a \\
      \end{matrix} \bigg|\bigg|
      \begin{matrix}
        b & b & 0 \\
        b & b & b \\
      \end{matrix} \bigg|(-1)^{2b}[2b+1]\\
      =& \sum_{a,b}(-1)^{2a}[2a+1](-1)^{2b}[2b+1]\frac{(-1)^{2a+2b}}{[2a+1][2b+1]}\frac{(-1)^{2b}}{[2b+1]}(-1)^{2b}[2b+1]\\
      =&\sum_{a,b}1,
  \end{split}
   \end{equation*}
   where $a,b$ run over all the elements of $I_r$ such that $(a,a,b)$ and $(b,b,b)$ are admissible. Counting the number of such pairs $(a,b),$ we have 
    $$\mathrm{TV}_r(T_{(2,6)})=\Big(\lfloor\frac{r-2}{3}\rfloor+1\Big)\Big(\lfloor\frac{2r-2}{3}\rfloor+1\Big).$$
\end{proof}

\begin{remark} Conjecture \ref{VC} can be generalized to non-hyperbolic $3$-manifolds by considering the Gromov norm, and Propositions  \ref{=1}, \ref{3_1}, \ref{Hopf}, \ref{T(2,4)}, \ref{T(2,6)} prove that for the corresponding cases.
\end{remark}

\subsection{Numerical evidence for Conjecture \ref{int}}\label{integrality}

In this subsection, we   provide further evidence for Conjecture \ref{int} by numerically calculating the Turaev-Viro invariants for the complements of the torus knots $T_{(2,5)},$ $T_{(3,5)},$ $T_{(2,7)},$ $T_{(3,7)},$ $T_{(2,9)}$ and $T_{(2,11)}$.

\subsubsection{ Knot $T_{(2,5)}$}

The table below contains the values of $\mathrm{TV}_r\big(T_{(2,5)}; e^{\frac{k\pi\sqrt{-1}}{r}}\big)$ for $k=1,2,3$ and $r\leqslant 20$. 
\bigskip

\centerline{\tiny\begin{tabular}{|c||c|c|c|c|c|c|c|c|c|c|c|c|c|c|c|c|c|c|c|c|c|} 
\hline
 $\ k\setminus r \ $ & $3$ & $4$ & $5$ & $6$ & $7$ & $8$ & $9$ & $10$  & $11$  & $12$ & $13$ & $14$ & $15$  & $16$ & $17$ & $18$ & $19$ & $20$ \\ 
\hline
&&&&&&&&&&&&&&&&&&\\
\ \ $1$ \ \ & \ \ $1$\ \ & \ \ $1$ \ \ &  \ \ $0.381966$ \ \  & \ \ $1$ \ \ & \ \ $3$ \ \  & \ \ $3$ \ \ & \ \ $2$ \ \ & \ \ $0.763932$ \ \ & \ \ $2 $ \ \ & \ \ $5$ \ \ & \ \ $5$\ \ & \ \ $3$ \ \ & \ \ $1.14590$ \ \  & \ \ $3$ \ \ & \ \ $7$ \ \  & \ \ $7$ \ \ & \ \ $4$ \ \ &\ \ $1.52786$ \ \ \\
\hline
&&&&&&&&&&&&&&&&&&\\
\ \ $2$ \ \ & \ \ $1$ \ \ & \ \ $$ \ \ & \ \ $2.61803$ \ \ & \ \ $$ \ \  & \ \ $3$ \ \ & \ \ $$ \ \ &   \ \ $2$ \ \ & \ \ $$ \ \ & \ \ $2$ \ \ & \ \ $$\ \ & \ \ $5$ \ \ & & \ \ $7.85410$ \ \ & \ \ $$ \ \  & \ \ $7$ \ \ & \ \ $$ \ \  &  \ \ $4$ \ \ &   \ \ $$ \ \  \\
\hline
&&&&&&&&&&&&&&&&&&\\
\ \ $3$ \ \ & \ \ $$\ \ & \ \ $1$ \ \ &  \ \ $2.61803$ \ \  & \ \ $$ \ \ & \ \ $3$ \ \  & \ \ $3$ \ \ & \ \ $$ \ \ & \ \ $5.23607$ \ \ & \ \ $2 $ \ \ & \ \ $$ \ \ & \ \ $5$\ \ & \ \ $3$ \ \ & \ \ $$ \ \  & \ \ $3$ \ \ & \ \ $7$ \ \  & \ \ $$ \ \ & \ \ $4$ \ \ &\ \ $10.4721$ \ \ \\
\hline
\end{tabular}}
\bigskip

\subsubsection{ Knot $T_{(3,5)}$}

The table below contains the values of $\mathrm{TV}_r\big(T_{(3,5)}; e^{\frac{k\pi\sqrt{-1}}{r}}\big)$ for $k=1,2,3$ and $r\leqslant 20$. 
\bigskip

\centerline{\tiny\begin{tabular}{|c||c|c|c|c|c|c|c|c|c|c|c|c|c|c|c|c|c|c|c|c|c|} 
\hline
 $\ k\setminus r \ $ & $3$ & $4$ & $5$ & $6$ & $7$ & $8$ & $9$ & $10$  & $11$  & $12$ & $13$ & $14$ & $15$  & $16$ & $17$ & $18$ & $19$ & $20$ \\ 
\hline
&&&&&&&&&&&&&&&&&&\\
\ \ $1$ \ \ & \ \ $1$\ \ & \ \ $1$ \ \ &  \ \ $0.381966$ \ \  & \ \ $1$ \ \ & \ \ $2$ \ \  & \ \ $3$ \ \ & \ \ $2$ \ \ & \ \ $1.38197$ \ \ & \ \ $2 $ \ \ & \ \ $4$ \ \ & \ \ $4$\ \ & \ \ $3$ \ \ & \ \ $ 1.76393$ \ \  & \ \ $3$ \ \ & \ \ $6$ \ \  & \ \ $6$ \ \ & \ \ $4$ \ \ &\ \ $ 2.14590$ \ \ \\
\hline
&&&&&&&&&&&&&&&&&&\\
\ \ $2$ \ \ & \ \ $1$ \ \ & \ \ $$ \ \ & \ \ $2.61803$ \ \ & \ \ $$ \ \  & \ \ $2$ \ \ & \ \ $$ \ \ &   \ \ $2$ \ \ & \ \ $$ \ \ & \ \ $2$ \ \ & \ \ $$\ \ & \ \ $4$ \ \ & & \ \ $6.23607$ \ \ & \ \ $$ \ \  & \ \ $6$ \ \ & \ \ $$ \ \  &  \ \ $4$ \ \ &   \ \ $$ \ \  \\
\hline
&&&&&&&&&&&&&&&&&&\\
\ \ $3$ \ \ & \ \ $$\ \ & \ \ $1$ \ \ &  \ \ $2.61803$ \ \  & \ \ $$ \ \ & \ \ $2$ \ \  & \ \ $3$ \ \ & \ \ $$ \ \ & \ \ $3.61803$ \ \ & \ \ $2 $ \ \ & \ \ $$ \ \ & \ \ $4$\ \ & \ \ $3$ \ \ & \ \ $ $ \ \  & \ \ $3$ \ \ & \ \ $6$ \ \  & \ \ $$ \ \ & \ \ $4$ \ \ &\ \ $ 8.85410$ \ \ \\
\hline
\end{tabular}}
\bigskip


\subsubsection{ Knot $T_{(2,7)}$}

The table below contains the values of $\mathrm{TV}_r\big(T_{(2,7)}; e^{\frac{k\pi\sqrt{-1}}{r}}\big)$ for $k=1,2,3$ and $r\leqslant 21$. 
\bigskip

\centerline{\tiny\begin{tabular}{|c||c|c|c|c|c|c|c|c|c|c|c|c|c|c|c|c|c|c|c|c|c|c|} 
\hline
 $\ k\setminus r\ $ & $3$ & $4$ & $5$ & $6$ & $7$ & $8$ & $9$ & $10$  & $11$  & $12$ & $13$ & $14$ & $15$  & $16$ & $17$ & $18$ & $19$ & $20$ & $21$ \\ 
\hline
&&&&&&&&&&&&&&&&&&&\\
\ \ $1$ \ \ & \ \ $1$\ \ & \ \ $1$ \ \ &  \ \ $2$ \ \  & \ \ $1$ \ \ & \ \ $0.307979$ \ \  & \ \ $1$ \ \ & \ \ $4$ \ \ & \ \ $3$ \ \ & \ \ $3 $ \ \ & \ \ $5$ \ \ & \ \ $2$\ \ & \ \ $0.615957$ \ \ & \ \ $2$ \ \  & \ \ $7$ \ \ & \ \ $5$ \ \  & \ \ $5$ \ \ & \ \ $8$ \ \ &\ \ $3$ \ \ & \ \ $0.923936$ \ \ \\
\hline
&&&&&&&&&&&&&&&&&&&\\
\ \ $2$ \ \ & \ \ $1$ \ \ & \ \ $$ \ \ & \ \ $2$ \ \ & \ \ $$ \ \  & \ \ $0.643104$ \ \ & \ \ $$ \ \ &   \ \ $4$ \ \ & \ \ $$ \ \ & \ \ $3$ \ \ & \ \ $$\ \ & \ \ $2$ \ \ & & \ \ $2$ \ \ & \ \ $$ \ \  & \ \ $5$ \ \ & \ \ $$ \ \  &  \ \ $8$ \ \ &   \ \ $$ \ \ & \ \ $1.92931$ \ \   \\
\hline
&&&&&&&&&&&&&&&&&&&\\
\ \ $3$ \ \ & \ \ $$\ \ & \ \ $1$ \ \ &  \ \ $2$ \ \  & \ \ $$ \ \ & \ \ $5.04892$ \ \  & \ \ $1$ \ \ & \ \ $$ \ \ & \ \ $3$ \ \ & \ \ $3 $ \ \ & \ \ $$ \ \ & \ \ $2$\ \ & \ \ $10.0978$ \ \ & \ \ $$ \ \  & \ \ $7$ \ \ & \ \ $5$ \ \  & \ \ $$ \ \ & \ \ $8$ \ \ &\ \ $3$ \ \ & \ \ $$ \ \ \\
\hline
\end{tabular}}
\bigskip


\subsubsection{ Knot $T_{(3,7)}$}

The table below contains the values of $\mathrm{TV}_r\big(T_{(3,7)}; e^{\frac{k\pi\sqrt{-1}}{r}}\big)$ for $k=1,2,3$ and $r\leqslant 21$. 
\bigskip

\centerline{\tiny\begin{tabular}{|c||c|c|c|c|c|c|c|c|c|c|c|c|c|c|c|c|c|c|c|c|c|c|} 
\hline
 $\ k\setminus r\ $ & $3$ & $4$ & $5$ & $6$ & $7$ & $8$ & $9$ & $10$  & $11$  & $12$ & $13$ & $14$ & $15$  & $16$ & $17$ & $18$ & $19$ & $20$ & $21$ \\ 
\hline
&&&&&&&&&&&&&&&&&&&\\
\ \ $1$ \ \ & \ \ $1$\ \ & \ \ $1$ \ \ &  \ \ $2$ \ \  & \ \ $1$ \ \ & \ \ $0.198062$ \ \  & \ \ $1$ \ \ & \ \ $3$ \ \ & \ \ $3$ \ \ & \ \ $3 $ \ \ & \ \ $4$ \ \ & \ \ $2$\ \ & \ \ $0.841166$ \ \ & \ \ $2$ \ \  & \ \ $2$ \ \ & \ \ $5$ \ \  & \ \ $5$ \ \ & \ \ $6$ \ \ &\ \ $3$ \ \ & \ \ $1.03923$ \ \ \\
\hline
&&&&&&&&&&&&&&&&&&&\\
\ \ $2$ \ \ & \ \ $1$ \ \ & \ \ $$ \ \ & \ \ $2$ \ \ & \ \ $$ \ \  & \ \ $3.24698$ \ \ & \ \ $$ \ \ &   \ \ $3$ \ \ & \ \ $$ \ \ & \ \ $3$ \ \ & \ \ $$\ \ & \ \ $2$ \ \ & & \ \ $2$ \ \ & \ \ $$ \ \  & \ \ $5$ \ \ & \ \ $$ \ \  &  \ \ $6$ \ \ &   \ \ $$ \ \ & \ \ $11.5429$ \ \   \\
\hline
&&&&&&&&&&&&&&&&&&&\\
\ \ $3$ \ \ & \ \ $$\ \ & \ \ $1$ \ \ &  \ \ $2$ \ \  & \ \ $$ \ \ & \ \ $1.55496$ \ \  & \ \ $1$ \ \ & \ \ $$ \ \ & \ \ $3$ \ \ & \ \ $3 $ \ \ & \ \ $$ \ \ & \ \ $2$\ \ & \ \ $1.86294$ \ \ & \ \ $$ \ \  & \ \ $2$ \ \ & \ \ $5$ \ \  & \ \ $$ \ \ & \ \ $6$ \ \ &\ \ $3$ \ \ & \ \ $$ \ \ \\
\hline
\end{tabular}}
\bigskip


\subsubsection{ Knot $T_{(2,9)}$}

The table below contains the values of $\mathrm{TV}_r\big(T_{(2,9)}; e^{\frac{k\pi\sqrt{-1}}{r}}\big)$ for $k=1,2,3$ and $r\leqslant 22$. 
\bigskip

\centerline{\tiny\begin{tabular}{|c||c|c|c|c|c|c|c|c|c|c|c|c|c|c|c|c|c|c|c|c|c|c|} 
\hline
 $\ k\setminus r\ $ & $3$ & $4$ & $5$ & $6$ & $7$ & $8$ & $9$ & $10$  & $11$  & $12$ & $13$ & $14$ & $15$  & $16$ & $17$ & $18$ & $19$ & $20$  & $21$ & $22$  \\ 
\hline
&&&&&&&&&&&&&&&&&&&&\\
\ \ $1$ \ \ & \ \ $1$\ \ & \ \ $1$ \ \ &  \ \ $1$ \ \  & \ \ $2$ \ \ & \ \ $3$ \ \  & \ \ $1$ \ \ & \ \ $0.283119$ \ \ & \ \ $1$ \ \ & \ \ $5$ \ \ & \ \ $4$ \ \ & \ \ $3$\ \ & \ \ $3$ \ \ & \ \ $5$ \ \  & \ \ $7$ \ \ & \ \ $2$ \ \  & \ \ $0.566237$ \ \ & \ \ $2$ \ \ & \ \ $9$\ \ & \ \ $7$ \ \  &  \ \ $5$ \ \  \\
\hline
&&&&&&&&&&&&&&&&&&&&\\
\ \ $2$ \ \ & \ \ $1$ \ \ & \ \ $$ \ \ & \ \ $1$ \ \ & \ \ $$ \ \  & \ \ $3$ \ \ & \ \ $$ \ \ &   \ \ $0.426022$ \ \ & \ \ $$ \ \ & \ \ $5$ \ \ & \ \ $$\ \ & \ \ $3$ \ \ & & \ \ $5$ \ \ & \ \ $$ \ \  & \ \ $2$ \ \ & \ \ $$ \ \  &  \ \ $2$ \ \ & \ \ $$ \ \  & \ \ $7$ \ \ &  \ \ $$ \ \  \\
\hline
&&&&&&&&&&&&&&&&&&&&\\
\ \ $3$ \ \ & \ \ $$\ \ & \ \ $1$ \ \ &  \ \ $1$ \ \  & \ \ $$ \ \ & \ \ $3$ \ \  & \ \ $1$ \ \ & \ \ $$ \ \ & \ \ $1$ \ \ & \ \ $5$ \ \ & \ \ $$ \ \ & \ \ $3$\ \ & \ \ $3$ \ \ & \ \ $$ \ \  & \ \ $7$ \ \ & \ \ $2$ \ \  & \ \ $$ \ \ & \ \ $2$ \ \ & \ \ $9$\ \ & \ \ $$ \ \  &  \ \ $5$ \ \  \\
\hline
\end{tabular}}
\bigskip

\subsubsection{ Knot $T_{(2,11)}$}

The table below contains the values of $\mathrm{TV}_r\big(T_{(2,11)}; e^{\frac{k\pi\sqrt{-1}}{r}}\big)$ for $k=1,2,3$ and $r\leqslant 22$. 
\bigskip

\centerline{\tiny\begin{tabular}{|c||c|c|c|c|c|c|c|c|c|c|c|c|c|c|c|c|c|c|c|c|c|c|c|c|c|c|} 
\hline
 $\ k\setminus r\ $ & $3$ & $4$ & $5$ & $6$ & $7$ & $8$ & $9$ & $10$  & $11$  & $12$ & $13$ & $14$ & $15$  & $16$ & $17$ & $18$ & $19$ & $20$ & $21$  & $22$  \\ 
\hline
&&&&&&&&&&&&&&&&&&&&\\
\ \ $1$ \ \ & \ \ $1$\ \ & \ \ $1$ \ \ &  \ \ $1$ \ \  & \ \ $1$ \ \ & \ \ $2$ \ \  & \ \ $3$ \ \ & \ \ $4$ \ \ & \ \ $1$ \ \ & \ \ $0.271554 $ \ \ & \ \ $1$ \ \ & \ \ $6$\ \ & \ \ $5$ \ \ & \ \ $4$ \ \  & \ \ $3$ \ \ & \ \ $3$ \ \  & \ \ $5$ \ \ & \ \ $7$ \ \ &\ \ $9$ \ \ & \ \ $2 $ \ \ & \ \ $0.543108$ \ \   \\
\hline
&&&&&&&&&&&&&&&&&&&&\\
\ \ $2$ \ \ & \ \ $1$\ \ & \ \ $$ \ \ &  \ \ $1$ \ \  & \ \ $$ \ \ & \ \ $2$ \ \  & \ \ $$ \ \ & \ \ $4$ \ \ & \ \ $$ \ \ & \ \ $0.353253$ \ \ & \ \ $$ \ \ & \ \ $6$\ \ & \ \ $$ \ \ & \ \ $4$ \ \  & \ \ $$ \ \ & \ \ $3$ \ \  & \ \ $$ \ \ & \ \ $7$ \ \ &\ \ $$ \ \ & \ \ $2$ \ \ & \ \ $$\ \  \\
\hline
&&&&&&&&&&&&&&&&&&&&\\
\ \ $3$ \ \ & \ \ $$\ \ & \ \ $1$ \ \ &  \ \ $1$ \ \  & \ \ $$ \ \ & \ \ $2$ \ \  & \ \ $3$ \ \ & \ \ $$ \ \ & \ \ $1$ \ \ & \ \ $0.582964$ \ \ & \ \ $$ \ \ & \ \ $6$\ \ & \ \ $5$ \ \ & \ \ $$ \ \  & \ \ $3$ \ \ & \ \ $3$ \ \  & \ \ $$ \ \ & \ \ $7$ \ \ &\ \ $9$ \ \ & \ \ $ $ \ \ & \ \ $1.16593$ \ \   \\
\hline
\end{tabular}}


\end{document}